\documentclass[10pt,a4paper]{article}

\makeatletter
\def\cl@chapter{}
\makeatother

\usepackage[LSBC3,T1]{fontenc}
\usepackage{graphicx}
\usepackage{amsfonts}
\usepackage{amssymb}
\usepackage{amsmath}
\usepackage{graphicx}
\usepackage{adjustbox}
\usepackage{enumitem}
\usepackage[top=2.5cm, bottom=3.2cm, left=2.4cm, right=2.4cm, footskip=1.5cm,footnotesep=1cm]{geometry}
\usepackage[colorlinks=true,urlcolor=blue,citecolor=blue]{hyperref}
\usepackage{natbib}
\usepackage{subfigure}
\usepackage{caption}
\usepackage[capitalise]{cleveref}
\usepackage{chessboard}
\usepackage{skak}

\usepackage{amsthm}
\newtheorem{theorem}{Theorem}
\newtheorem{definition}{Definition}
\newtheorem{proposition}{Proposition}

\newtheorem{lemma}{Lemma}

\newtheorem{remark}{Remark}
\newtheorem{example}{Example}
\newtheorem{exercise}{Exercise}

\Crefname{enumi}{}{}
\setlist[enumerate,1]{label=(\roman*)}
\crefname{exercise}{Exercise}{Exercises}

\newcommand{\qede}{\hspace*{\fill}$\Diamond$\medskip}

\newcommand*{\doi}[1]{DOI: \href{http://dx.doi.org/#1}{#1}}

\newcommand{\Fix}{\operatorname{Fix}}
\newcommand{\Id}{\operatorname{Id}}
\newcommand{\tto}{\rightrightarrows}
\newcommand{\wto}{\rightharpoonup}

\DeclareMathAlphabet\mbc{OMS}{cmsy}{b}{n}

\newcommand{\Hi}{\mathcal{H}}
\newcommand{\Eu}{E}
\newcommand{\R}{\mathbb{R}}
\newcommand{\B}{\mathbb{B}}
\newcommand{\Rnn}{\mathbb{R}^{n\times n}}
\newcommand{\bin}{\{0,1\}}
\newcommand{\Bnn}{\bin^{n\times n}}

\newcommand{\bs}{\boldsymbol}

\begin{document}

\title{\textbf{The Douglas--Rachford Algorithm for Convex and Nonconvex Feasibility Problems}}

\author{Francisco J. Arag\'on Artacho\thanks{Department of Mathematics, University of Alicante, \textsc{Spain}. Email:~\href{mailto:francisco.aragon@ua.es}{francisco.aragon@ua.es}}
	   \and
	   Rub\'en Campoy\thanks{Department of Mathematics, University of Alicante, \textsc{Spain}.
	   	                     Email:~\href{mailto:ruben.campoy@ua.es}{ruben.campoy@ua.es}}
	   \and
	   Matthew K. Tam\thanks{Inst.\ for Num.\ and Appl.\ Math.,
	   	                     University of G\"ottingen, \textsc{Germany}.
	   	                     Email:~\href{mailto:m.tam@math.uni-goettingen.de}{m.tam@math.uni-goettingen.de}}}

\maketitle

\begin{abstract}
  The Douglas--Rachford method, a projection algorithm designed to solve continuous optimization problems, forms the basis of a useful heuristic for solving combinatorial optimization problems. In order to successfully use the method, it is necessary to formulate the problem at hand as a feasibility problem with constraint sets having efficiently computable nearest points. In this self-contained tutorial, we develop the convergence theory of projection algorithms within the framework of fixed point iterations, explain how to devise useful feasibility problem formulations, and demonstrate the application of the Douglas--Rachford method to said formulations. The paradigm is then illustrated on two concrete problems: a generalization of the ``eight queens puzzle'' known as the ``$(m,n)$-queens problem'', and the problem of constructing a probability distribution with prescribed moments.

\paragraph*{Keywords} {Projection methods $\cdot$ Douglas--Rachford $\cdot$ Feasibility problem $\cdot$ Eight queens problem}
\paragraph*{MSC2010}{65K05 
          $\cdot$
          90C27 
          $\cdot$
          90-01 
          $\cdot$
          65-01 
         }

\end{abstract}

\section{Introduction}
The so-called \emph{feasibility problem} asks for a point contained in the intersection of a finite collection of constraint sets. Precisely, given a family of sets $C_1,C_2,\ldots,C_r$ contained in a Euclidean space $E$, the corresponding {feasibility problem} takes the form
\begin{equation}\label{eq:FeasibilityProblem}
\text{Find } x\in C:= \bigcap_{i=1}^r C_i.
\end{equation}
A feasibility problem is said to be \emph{consistent} when it has solution (i.e.\ when $\cap_{i=1}^r C_i\neq\emptyset$), otherwise it is said to be \emph{inconsistent}.
This seemingly simple problem provides a modeling framework with great flexibility and power. For the purpose of numerical schemes however, devising computationally tractable formulations is often a nontrivial task and some creativity is required.

In many situations of practical interest, it is difficult to find a point in $C$ directly. On the other hand, the individual constraint sets, $C_i$, can often be chosen to have relatively simple structure. \emph{Projection methods} are a family of iterative algorithms for solving  \eqref{eq:FeasibilityProblem} which aim to exploit this observation with a solution to the problem being obtained in the limit. In this context, ``simple'' is understood in the sense of easy-to-compute \emph{projection mappings}.

\begin{definition}[Projection mapping]\label{def:proj}
Given a nonempty subset $C\subseteq\Eu$, the \emph{projection mapping} (or \emph{projector}) onto $C$ is the possibly set-valued operator, $P_C:\Eu\tto C$, defined at each $x\in\Eu$ by
\begin{equation*}
P_C(x):=\left\{p\in C : \|x-p\|=d_C(x):=\inf_{c\in C}\|c-x\|\right\}.
\end{equation*}
\end{definition}

It is a straightforward exercise to verify that the projection mapping $P_C$ has nonempty values (i.e.\ $P_C(x)\neq\emptyset$ for all $x\in\Eu$) when $C$ is nonempty and closed. If $P_C(x)$ is a singleton set for all $x\in\Eu$, then $C$ is said to be \emph{Chebyshev}. In this case, the projector $P_C$ is a single-valued mapping that sends each point $x\in\Eu$ to its unique nearest point in $C$. Amongst other names, the projector is sometimes also called the \emph{projection operator}, the \emph{metric projection}, \emph{nearest point mapping} or the \emph{best approximation operator}.

\begin{exercise}
  Let $C\subseteq\Eu$ be a nonempty, closed set. Show that $P_C(x)\neq\emptyset$ for all $x\in\Eu$.
\end{exercise}	

The simplest projection algorithm for solving \eqref{eq:FeasibilityProblem} is the \emph{method of cyclic projections} which iterates by successively applying the projectors onto each of the constraint sets. More precisely, given an initial point $x_0\in\Eu$, it generates a sequence ${(x_k)}_{k=0}^\infty$ according to
\begin{equation}\label{eq:AP_iteration}
x_{k+1} \in \left(P_{C_r} P_{C_{r-1}} \cdots  P_{C_1}\right)(x_k),\quad\text{for } k=0,1,2,\ldots,
\end{equation}
where we note that, for an operator $T$ and a set $C$, $T(C)$ is defined by $T(C):=\cup_{x\in C}T(x)$.

In the case that the sets $C_1,\dots,C_r$ are Chebyshev, \eqref{eq:AP_iteration} can be written with equality. However, in general, the inclusion is required due to potential set-valuedness of the projection mapping. In such cases, we will be interested in any sequence ${(x_k)}_{k=0}^\infty$ satisfying \eqref{eq:AP_iteration}. Two illustrations of the method, where the limit of ${(x_k)}_{k=0}^\infty$ is contained in the intersection, are provided in \Cref{fig:APscenarios}.

\begin{figure}[ht!]
\centering
\subfigure[The method of cyclic projections for two halfspaces]{\includegraphics[width=0.45\textwidth]{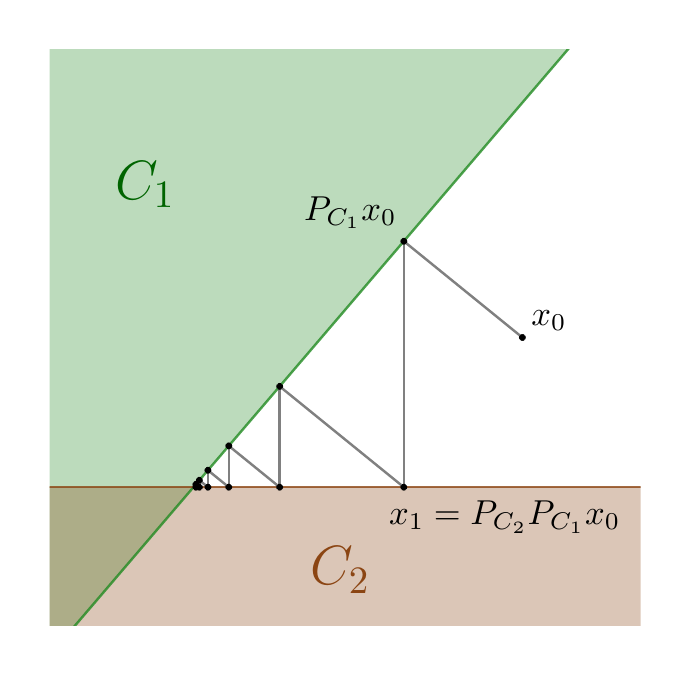}}\hspace{0.05\textwidth}
\subfigure[The method of cyclic projections for three balls]{\includegraphics[width=0.45\textwidth]{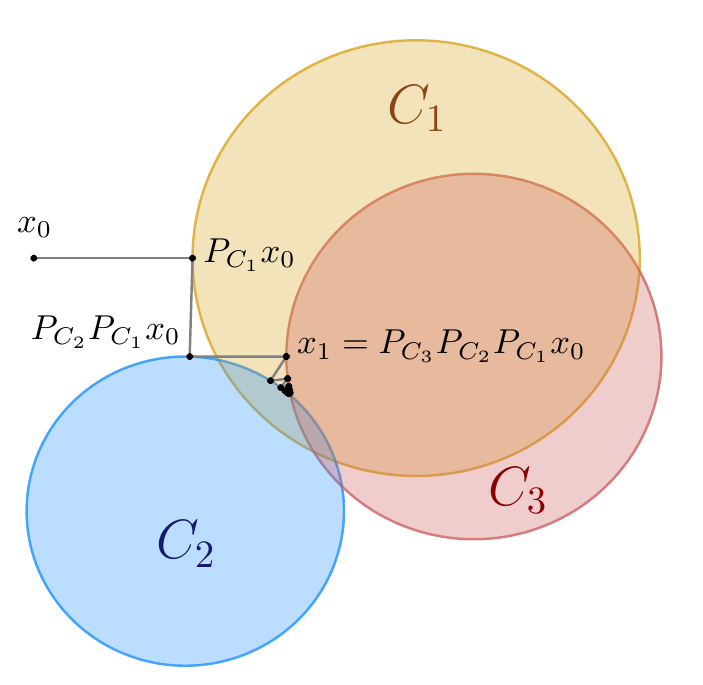}}
\caption{The method of cyclic projections applied to families of closed and convex sets with nonempty intersections.}\label{fig:APscenarios}
\end{figure}

The origins of the method of cyclic projections can be traced back at least to the work of \cite{VN50}, who proved convergence of  \eqref{eq:AP_iteration} when $r=2$ and the sets are subspaces. \cite{H62} subsequently showed that the result remains true for $r>2$. \cite{Breg65} established convergence of the method when the constraints are closed and convex sets.

The following familiar example illustrates the potential advantages of projection methods and the feasibility problem paradigm.
\begin{example}[Linear systems]\label{example:linear_systems}
Let $A:\mathbb{R}^n\to\mathbb{R}^m$ be a linear map and $b\in\mathbb{R}^m$. Consider the problem of solving the linear equation $Ax=b$.
By letting $a_i^T$ denote the $i$th row of the matrix $A$, this can be expressed as
   $$ \begin{pmatrix} a_1^Tx \\ \vdots \\ a_m^Tx \\ \end{pmatrix} = \begin{pmatrix} a_1^T \\  \vdots \\ a_m^T \\ \end{pmatrix}x = Ax = b = \begin{pmatrix} b_1 \\  \vdots \\ b_m \end{pmatrix}. $$
In other words, $Ax=b$ if and only if
    \begin{equation}
    \label{eq:linear_system}
    x\in\bigcap_{i=1}^mC_i\text{~~\color{black}where~~}C_i:=\left\{x\in\mathbb{R}^n:a_i^Tx=b_i\right\}.
    \end{equation}
The projector onto the hyperplane $C_i$ is given by
\begin{equation}\label{eq:proj_hyperplane}
P_{C_i}(x) = x + \left(b_i-a_i^Tx\right)\frac{a_i}{\|a_i\|^2},
\end{equation}
which can be easily computed using only vector arithmetic. The well-known \emph{Kaczmarz's method} (\cite{K37}) is precisely the method of cyclic projections applied to \eqref{eq:linear_system}.

One utility of viewing this problem as a feasibility problem is that additional constraints can be easily incorporated without changing the algorithm used. For instance, a nonnegative solution to the linear system $Ax=b$ can be found by augmenting \eqref{eq:linear_system} with the constraint
  $ C_{m+1} := \{x\in\mathbb{R}^n:x\geq 0\}, $
whose projection is given by $P_{C_{m+1}}(x)=\max\{0,x\}$ (understood in the pointwise sense).

Another possible variant of the problem involves considering the linear inequality system $Ax\leq b$. This too can be cast in terms of the feasibility problem \eqref{eq:linear_system} with constraint sets $C_i:=\{x\in\mathbb{R}^n:a_i^Tx\leq b_i\}$. The projection onto these sets are given by
\begin{equation}\label{eq:proj_halfspace}
P_{C_i}(x) = x + \min\left\{0,b_i-a_i^Tx\right\}\frac{a_i}{\|a_i\|^2},
\end{equation}
which also only requires vector arithmetic.\qede
\end{example}

This manuscript is intended to be a condensed tutorial on the Douglas--Rachford algorithm for solving convex and nonconvex feasibility problems for the unacquainted. For general reviews on projection algorithms, we refer the reader to~\cite{BC17,C12,C84,CC15,D01}. The remainder of this tutorial is organized as follows. In~\Cref{sec:Preliminaries}, we discuss mathematical preliminaries including general properties of nonexpansive operators, fixed point iterations and projection mappings. In~\Cref{sec:ProjAlg}, we use these preliminaries to analyze the convergence behavior of numerous projection methods including the Douglas--Rachford algorithm. In~\Cref{sec:Nonconvex}, we apply the Douglas--Rachford method to solve a combinatorial problem known as the \emph{$(m,n)$-queens puzzle}. For simplicity, this tutorial mainly focuses on the finite-dimensional setting, although appropriately modified, most of the results presented remain true in infinite dimensions. Indeed, the technical difficulties associated with the infinite-dimensional setting are explained in \Cref{sec:infinite}, where the problem of constructing a probability distribution with prescribed moments is also discussed. We finish with some conclusions and open questions
in~\Cref{sec:conslusions}.

\section{Preliminaries}\label{sec:Preliminaries}

Throughout this paper, $\Eu$ denotes a \emph{Euclidean space} equipped with inner product~$\langle\cdot,\cdot\rangle$ and induced norm~${\|\cdot\|}$.
Given a non-empty set $D\subseteq\Eu$, $T:D\tto\Eu$ denotes a \emph{set-valued operator} that maps a point in $D$ to a subset of $\Eu$ (i.e.\ $T(x)\subseteq \Eu$ for all $x\in D$). In the case when $T(x)=\{u\}$ for all $x\in D$, $T$ is said to be a \emph{single-valued mapping}, which is denoted as $T:D\to\Eu$. In an abuse of notation, we write $T(x)=u$ whenever $T(x)=\{u\}$. The set of \emph{fixed points} of an operator $T$, denoted by $\Fix T$, is given by
$$\Fix T:=\left\{x\in\Eu \mid x\in T(x)\right\}.$$
The \emph{identity operator} is the mapping $\Id:\Eu\to\Eu$ that maps every point to itself.

\subsection{Convergence of fixed point algorithms}
The mathematical analysis of projection algorithms can be performed within the framework of \emph{fixed-point theory}. In this section, we establish general results concerning convergence of \emph{Banach--Picard iterations}, which we will later specialize to projection algorithms.

We take the following abstract representation of an iterative algorithm: Given a mapping $T:\Eu\tto\Eu$ and a point $x_0\in\Eu$, we consider the scheme defined by
\begin{equation}\label{eq:Picard_iteration}
x_{k+1}\in T(x_k),  \quad \text{for } k=0,1,2,\ldots.
\end{equation}
We refer to \eqref{eq:Picard_iteration} as the \emph{fixed point iteration} or \emph{Banach--Picard iteration} defined by $T$. Within this setting, the operator $T$ is an abstract representation of the algorithm and its fixed points are assumed to provide solutions to the problem at hand. A classical result due to \cite{banach1922operations}, known as \emph{Banach's contraction principle}, is perhaps the most well-known criteria for ensuring convergence of \eqref{eq:Picard_iteration}.

 \begin{theorem}[Banach's contraction principle]\label{th:banach contraction}
 Let $T\colon \Eu\to \Eu$ be a contraction mapping, i.e., there exists $\kappa\in[0,1[$ such that
 \begin{equation}
 \label{eq:contraction}
 \|T(x)-T(y)\|\leq\kappa\|x-y\|\quad\forall x,y\in\Eu.
 \end{equation}
 Then $T$ has exactly one fixed point $x^\star\in\Eu$. Furthermore, for any $x_0\in X$ the sequence $x_{k+1}=T(x_k)$ converges to $x^\star$ with the linear rate $\kappa$; i.e., $\|x_k-x^\star\|\leq \kappa^k \|x_0-x^\star\|$ for all $k=0,1,2,\ldots$.
 \end{theorem}

For the purposes of projection algorithms for feasibility problem, Banach's contraction principle is too restricted. We will rarely be fortunate enough for the operator $T$ to be contractive and it will usually be the case that our problem has more than one solution (i.e., $\Fix T$ is not singleton). On the other hand, nonexpansivity of $T$ (i.e., \eqref{eq:contraction} with $\kappa=1$) is also not sufficient to ensure convergence of its Banach--Picard iteration. For a simple example, consider the mapping $T=-\Id$ and any $x_0\neq 0$. Therefore, notions stronger than nonexpansivity but weaker than contractivity are required. In what follows, we explore such notions.

\begin{definition}[Notions of nonexpansiveness]
	Let $D$ be a nonempty subset of $\Eu$ and let $T:D\to\Eu$. The operator $T$ is said to be
	\begin{enumerate}
		\item \emph{nonexpansive} if
		\begin{equation*}
		\|T(x)-T(y)\|\leq\|x-y\|, \quad \forall x,y\in D;
		\end{equation*}
		\item \emph{firmly nonexpansive} if
		\begin{equation*}
		\|T(x)-T(y)\|^2+\|(\Id-T)(x)-(\Id-T)(y)\|^2\leq\|x-y\|^2, \quad \forall x,y\in D,
		\end{equation*}
		or, equivalently,
		\begin{equation}\label{eq:alt fne}
		\langle x-y,T(x)-T(y)\rangle\geq \|T(x)-T(y)\|^2, \quad \forall x,y\in D;
		\end{equation}
		\item $\alpha$\emph{-averaged} for  $\alpha\in\,]0,1[$ if there exists a nonexpansive operator $R:D\to\Eu$~such~that
		\begin{equation*}
		T=(1-\alpha)\Id+\alpha R.
		\end{equation*}
	\end{enumerate}
\end{definition}

It follows immediately from the respective definitions that both firm nonexpansiveness and $\alpha$-averaged\-ness imply nonexpansiveness, while the converse implications are not true in general. Moreover, it also follows that if $T$ is $\alpha$-averaged, then it is also $\alpha'$-averaged for all $\alpha'\in[\alpha,1[$. Further relationships between the notions are shown in the following proposition. For more, see, e.g.,~\cite[Chapter~4]{BC17}.

\begin{proposition}\label{fact:firmly_nonexp}
	Let $D\subseteq\Eu$ be nonempty and let $T:D\to\Eu$. The following hold:
	\begin{enumerate}
		\item If $T$ is $\alpha$-averaged for $\alpha\in{]0,1[}$, then
		$$\|T(x)-T(y)\|^2\leq \|x-y\|^2-\frac{1-\alpha}{\alpha}\|(\Id-T)(x)-(\Id-T)(y)\|^2,\quad\forall x,y,\in D.$$\label{fact:firmly_nonexp0}
		\item $T$ is firmly nonexpansive if and only if $2T-\Id$ is nonexpansive.\label{fact:firmly_nonexpI}
		\item If $T$ is $\alpha$-averaged with $\alpha\in\left]0,\frac{1}{2}\right]$, then $T$ is firmly nonexpansive.\label{fact:firmly_nonexpII}
	\end{enumerate}
\end{proposition}
\begin{proof}
\Cref{fact:firmly_nonexp0}~By definition, there exists a nonexpansive operator $R:D\to\Eu$ such that $T=(1-\alpha)\Id+\alpha R$. We therefore have that
$$ \|T(x)-T(y)\|^2 = (1-\alpha)^2\|x-y\|^2 + 2\alpha(1-\alpha)\langle x-y,R(x)-R(y)\rangle + \alpha^2\|R(x)-R(y)\|^2, $$
and that
\begin{align*}
\frac{1}{\alpha}\|(\Id-T)(x)-(\Id-T)(y)\|^2 &= \alpha\|(\Id-R)(x)-(\Id-R)(y)\|^2 \\
&= \alpha\left(\|x-y\|^2 - 2\langle x-y,R(x)-R(y)\rangle + \|R(x)-R(y)\|^2\right).
\end{align*}
By summing these two equalities together, we obtain
\begin{align*}
\|T(x)-T(y)\|^2+\frac{1-\alpha}{\alpha}\|(\Id-T)(x)-(\Id-T)(y)\|^2
&= (1-\alpha)\|x-y\|^2 + \alpha\|R(x)-R(y)\|^2.
\end{align*}
The result then follows by using the nonexpansivity of $R$.
\Cref{fact:firmly_nonexpI}~Left as an exercise for the reader.
\Cref{fact:firmly_nonexpII}~Follows immediately from \Cref{fact:firmly_nonexp0}.
\end{proof}

\begin{exercise}
	Prove \Cref{fact:firmly_nonexp}\Cref{fact:firmly_nonexpI}.
\end{exercise}	

Our next result provides a useful criteria for convergence of the Banach--Picard iteration in~\eqref{eq:Picard_iteration}; namely, it suffices to assume that $T$ is averaged and possesses a fixed point. This result is actually a special case of \emph{Opial's Theorem} (\cite{opial1967weak}), which remains true in Banach spaces satisfying \emph{Opial's property}.
\begin{theorem}[Convergence of averaged iterations]\label{th:Opial}
Let $D$ be a nonempty, closed and convex subset of $\Eu$ and let $T:D\to D$ be an $\alpha$-averaged operator with $\Fix T\neq\emptyset$. Given any $x_0\in D$, set
	$$
	x_{k+1}=T(x_k), \quad\text{for } k=0,1,2\ldots.
	$$
	Then ${(x_k)}_{k=0}^\infty$ converges to a point $x^\star\in\Fix T$.
\end{theorem}
\begin{proof}
Let $x\in\Fix T$. Since $T$ is $\alpha$-averaged, by~\Cref{fact:firmly_nonexp}\Cref{fact:firmly_nonexp0}, we have
	  $$\|x_{k+1}-x\|^2 + \frac{1-\alpha}{\alpha}\|x_k-x_{k+1}\|^2 \leq \|x_k-x\|^2\quad\text{for all } k=0,1,2,\ldots.$$
It then follows that $(\|x_k-x\|)_{k=1}^\infty$ is non-increasing, $(x_k)_{k=0}^\infty$ is bounded and $x_k-x_{k+1}\to 0$. Now, as a bounded sequence, $(x_k)_{k=0}^\infty$ has a cluster point $x^\star$. Let $(x_{k_n})_{n=0}^{\infty}$ be a subsequence such that $x_{k_n}\to x^\star$. We therefore have that
  $$ (\Id-T)(x^\star) = \lim_{n\to\infty}(\Id-T)(x_{k_n}) = \lim_{n\to\infty}(x_{k_n}-x_{k_n+1})=0, $$
which shows that $x^\star\in\Fix T$. Since $(\|x_k-x^\star\|)_{k=1}^\infty$ is non-increasing and $\|x_{k_n}-x^\star\|\to 0$, it follows that $\|x_k-x^\star\|\to 0$, which completes the proof.
\end{proof}

Clearly, a Banach--Picard iteration~\eqref{eq:Picard_iteration} can only converge when the operator $T$ possesses at least one fixed point. Nevertheless, when $\Fix T=\emptyset$, it is still possible to say something about the asymptotic behavior when $T$ is averaged. In this tutorial, we prove a simplified version of the asymptotic behavior known from (\cite{P71,BBR78}). In order to do so, we first recall the following \emph{summability method}.

\begin{lemma}[Ces\'aro summation]\label{lem:cesaro}
Let $(x_k)\subseteq\Eu$ be a convergent sequence with limit $x^\star\in\Eu$. Then
 $$\frac{1}{k}\sum_{n=1}^kx_n\to x^\star\text{~~as~~}k\to\infty.$$
\end{lemma}
\begin{proof}
Left as an exercise for the reader.
\end{proof}

\begin{exercise}
	Prove \Cref{lem:cesaro} and give an example of a non-convergent sequence $(x_k)\subseteq\Eu$ which is \emph{Ces\'aro summable}, i.e., $\frac{1}{k}\sum_{n=1}^kx_n\to x^\star$ for some $x^\star\in\Eu$. Is a Ces\'aro summable sequence necessarily bounded?
\end{exercise}	

The following lemma will help to streamline the proof of \Cref{fact:pazy}

\begin{lemma}\label{lem:asympt_averaged}
	Suppose $T\colon\Eu\to\Eu$ is an $\alpha$-averaged operator for some $\alpha\in\,]0,1[$. Let $x_0,y_0\in\Eu$ and denote $ x_k=T^k(x_0)$ and $y_k=T^k(y_0)$ for $k=1,2,\ldots$. Then
	\begin{equation} \frac{1-\alpha}{\alpha}\sum_{k=0}^n\|(\Id-T)(x_k)-(\Id-T)(y_k)\|^2 \leq \|x_0-y_0\|^2-\|x_{n+1}-y_{n+1}\|^2,\quad\text{for all } n=0,1,2,\ldots.
	\end{equation}
	In particular,  $(\Id-T)(x_k)-(\Id-T)(y_k) \to 0$ as $k\to\infty$.
\end{lemma}
\begin{proof}
Since $T$ is $\alpha$-averaged, \Cref{fact:firmly_nonexp}\Cref{fact:firmly_nonexp0} implies that
  $$ \|x_{k+1}-y_{k+1}\|^2 + \frac{1-\alpha}{\alpha}\|(\Id-T)(x_k)-(\Id-T)(y_k)\|^2 \leq \|x_k-y_k\|^2,\quad\text{for all } k=0,1,2,\ldots,$$
which telescopes to yield
  $$\|x_{n+1}-y_{n+1}\|^2+\frac{1-\alpha}{\alpha}\sum_{k=0}^n\|(\Id-T)(x_k)-(\Id-T)(y_k)\|^2 \leq \|x_0-y_0\|^2,\quad\text{for all } n=0,1,2,\ldots.$$
The series in the previous equation therefore converges as $n\to\infty$, from which the result follows.
\end{proof}

\begin{theorem}[Asymptotic behavior of averaged iterations]\label{fact:pazy}
	Let $T:\Eu\to\Eu$ be an $\alpha$-averaged operator for some $\alpha\in\,]0,1[$. Then
	$$ \Fix T=\emptyset \quad\iff\quad \|T^k(x)\|\to\infty \text{~~for any~}x\in\Eu.$$
\end{theorem}
\begin{proof}
To prove the result, we establish its contrapositive: $\Fix T\neq\emptyset $ if and only if there exists an $x\in\Eu$ such that $\|T^k(x)\|\not\to\infty$. To this end, first take any $x\in\Eu$ and suppose that $\Fix T\neq\emptyset$. Then the sequence $T^k(x)$ is convergent by \Cref{th:Opial} and hence, in particular, also bounded.

To prove the reverse implication, suppose there exists an $x\in\Eu$ such that $x_k=T^k(x)$ contains a bounded subsequence. The sequence $(x_k)$ then possesses a convergent subsequence, say $x_{k_n}\to x^\star$. Since $T$ is $\alpha$-averaged, applying \Cref{lem:asympt_averaged} with $x_0=x$ and $y_0=x^\star$ yields
  $$ (\Id-T)(x_k)-(\Id-T)T^k(x^\star)\to 0 \implies (\Id-T)T^{k_n}(x^\star)\to (\Id-T)(x^\star). $$
Also, applying  \Cref{lem:asympt_averaged} with $x_0=x^\star$ and $y_0=T(x^\star)$ yields
$$\frac{1-\alpha}{\alpha}\sum_{j=1}^{k}\|(\Id-T)T^{j-1}(x^\star)-(\Id-T)T^j(x^\star)\|^2 \leq \|(\Id-T)(x^\star)\|^2- \|(\Id-T)T^{k}(x^\star)\|^2. $$
Taking the limit along the subsequence $(k_n)$ gives
$$\frac{1-\alpha}{\alpha}\sum_{j=1}^{\infty}\|(\Id-T)T^{j-1}(x^\star)-(\Id-T)T^j(x^\star)\|^2 \leq \|(\Id-T)(x^\star)\|^2- \|(\Id-T)(x^\star)\|^2=0, $$
which shows that all terms in the summation are identically zero.  Consequently, we have
$$ (\Id-T)T^j(x^\star)=(\Id-T)(x^\star)\implies (\Id-T)(x_k)\to(\Id-T)(x^\star). $$ Denoting $x_0:=x$, we observe that
$$ x_0 = \sum_{n=0}^{k-1}(\Id-T)(x_n) + x_k \implies \frac{1}{k}x_0 = \frac{1}{k}\sum_{n=0}^{k-1}(\Id-T)(x_n) + \frac{1}{k}x_k. $$
Taking the limit along the subsequence $x_{k_n}$ and using \Cref{lem:cesaro} gives
$$ 0\cdot x_0 = (\Id-T)(x^*) + 0\cdot x^\star \implies x^\star = T(x^\star).  $$
In order words, $x^\star\in\Fix T$ which shows that $\Fix T\neq\emptyset$.  The proof is now complete.
\end{proof}

\subsection{The projection mapping}
In this section, we examine  properties of the projector onto closed and convex sets.
\begin{proposition}\label{fact:projection}
	Let $C\subseteq\Eu$ be nonempty, closed and convex. Then $C$ is a Chebyshev set and, for every $x\in\Eu$, we have
		\begin{equation}\label{eq:projection}
		p=P_C(x) \quad \iff \quad p\in C  \text{~~and~~}  \langle c-p,x-p \rangle \leq 0 \text{~~for all~} c\in C.
		\end{equation}
\end{proposition}
\begin{proof}
Let $p,q\in P_C(x)$, which is nonempty because $C$ is closed. By the parallelogram law and the definition of the projector, we have
\begin{align*}
\|p-q\|^2&=\|(p-x)-(q-x)\|^2\\
&=2\|p-x\|^2+2\|q-x\|^2-\|p+q-2x\|^2\\
&=4d_C^2(x)-4\left\|\frac{p+q}{2}-x\right\|^2\leq 0,
\end{align*}
where the last inequality follows because convexity of $C$ yields $(p+q)/2\in C$. This implies $p=q$, so we conclude that the projection must be unique.

To prove the forward implication in~\eqref{eq:projection}, let $x\in E$ and $p=P_C(x)$. Pick any $c\in C$, and define $z_\lambda:=(1-\lambda)p+\lambda c\in C$, for all $\lambda\in {]0,1[}$. Then,
\begin{align*}
\|x-p\|^2\leq \|x-z_\lambda\|^2&=\|x-p-\lambda(c-p)\|^2\\
&=\|x-p\|^2+\lambda^2\|c-p\|^2-2\lambda\langle x-p,c-p\rangle.
\end{align*}
Since $\lambda>0$, we deduce $\langle x-p,c-p\rangle\leq \frac{\lambda}{2}\|c-p\|^2$, from which $\langle x-p,c-p\rangle\leq 0$ follows by letting $\lambda\rightarrow 0$. The backward implication in~\eqref{eq:projection} is left as an exercise for the reader.
\end{proof}

\begin{exercise}
	Complete the proof of \Cref{fact:projection}. That is, prove the backward implication of \eqref{eq:projection}.
\end{exercise}	

\begin{exercise}[Affine projectors]
A non-empty set $C\subseteq\Eu$ is an \emph{affine subspace} if $\lambda c_1+(1-\lambda)c_2\in C$ for all $\lambda\in\mathbb{R}$ and $c_1,c_2\in C$. Show that the projection onto an affine subspace satisfies~\eqref{eq:projection} with equality. Use this equality to verify the formula for the projection onto the hyperplane given in \Cref{example:linear_systems}.
\end{exercise}

In the subsequent sections, it will sometimes be easier to work with the following operator, which is defined in terms of the projector.
\begin{definition}[Reflection mapping]
Given a nonempty subset $C\subseteq\Eu$, the \emph{reflection mapping} (or \emph{reflector}) with respect to $C$ is the operator $R_C:\Eu\tto\Eu$ given by
$$R_C:=2P_C-\Id.$$
Each element $r\in R_C(x)$ is called a \emph{reflection} of $x$ with respect to $C$.
\end{definition}
It is immediate from the definition that the reflector is single-valued precisely when the corresponding projector is. Illustrations of reflection mappings are shown in \Cref{fig:projector_example}.

\begin{figure}[ht!]
\centering
\subfigure[A closed and convex set $C$ (Chebyshev). The projection and reflection of the point $x$ are unique]{\includegraphics[width=0.4\textwidth]{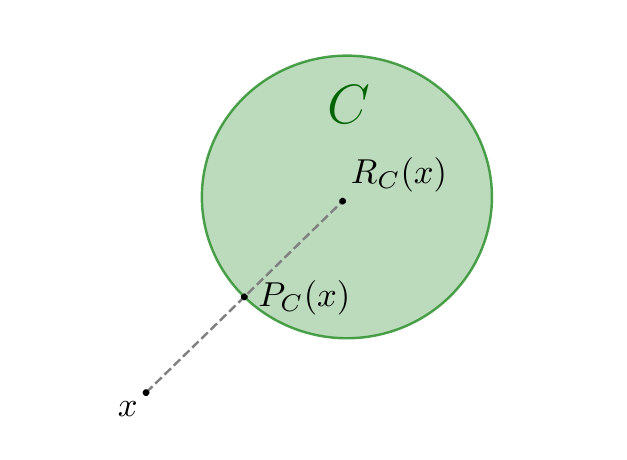}}\hspace{0.05\textwidth}
\subfigure[A closed but nonconvex set $C$. The projector and reflector are multivalued at $x$, with $P_C(x)=\{p_1,p_2\}$~and~$R_C(x)=\{r_1,r_2\}$]{{\includegraphics[width=0.4\textwidth]{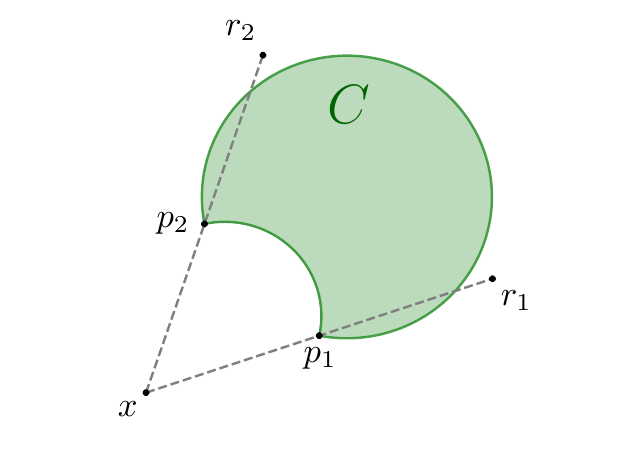}}}
\caption{Examples of projectors and reflectors onto convex and nonconvex sets.}\label{fig:projector_example}
\end{figure}

We conclude this section with the following fundamental result, which is responsible for the good behavior of projection algorithms in the convex setting.

\begin{proposition}\label{fact:projection_firmly_nonexapnsive}
Let $C\subseteq\Eu$ be nonempty, closed and convex. Then the projector, $P_C$, is firmly nonexpansive and the reflector, $R_C$, is nonexpansive.
\end{proposition}
\begin{proof}
Let $x,y\in\Eu$. By applying~\Cref{fact:projection} to $P_C(x)$ and $P_C(y)$, respectively, we obtain
\begin{align*}
\langle P_C(y)-P_C(x), x-P_C(x) \rangle \leq 0 \quad\text{and}\quad
\langle P_C(x)-P_C(y), y-P_C(y) \rangle \leq 0.
\end{align*}
The addition of these inequalities gives
$$\langle P_C(x)-P_C(y),P_C(x)-P_C(y)-(x-y)\rangle \leq 0,$$
which, by \eqref{eq:alt fne}, implies that $P_C$ is firmly nonexpansive. The fact that $R_C$ is nonexpansive now follows from \Cref{fact:firmly_nonexp}\Cref{fact:firmly_nonexpI}.
\end{proof}

Although projectors onto convex sets are firmly nonexpansive by \Cref{fact:projection_firmly_nonexapnsive}, in general, their compositions need not be.
\begin{example}
Let $\Eu=\R^2$. Consider the constraint sets $C_1:=\R\times\{0\}$ and $C_2:=\{(\lambda,\lambda):\lambda\in\mathbb{R}\}$. Then
	$P_{C_2}P_{C_1}(x,y)=P_{C_2}(x,0)=(x,x)/2$
	and hence
	$$\langle (1,-2)-(0,0),P_{C_2}P_{C_1}(1,-2)-P_{C_2}P_{C_1}(0,0)\rangle=-1/2.$$
Consequently, the composition $P_{C_2}P_{C_1}$ is not firmly nonexpansive, according to \eqref{eq:alt fne}.\qede
\end{example}

\begin{exercise}[Further examples of projectors]
By appealing to \Cref{fact:projection}, show that each of the following projector formulae holds.
\begin{enumerate}
\item (Closed balls) Suppose $r>0$ and denote $\B_r=\{x\in\Eu:\|x\|\leq r\}$. Then
  $$ P_{\B_r}(x) = \frac{r}{\max\{\|x\|,r\}}x. $$
\item (Nonnegative orthant) Suppose $\Eu=\mathbb{R}^n$ and denote $\mathbb{R}^n_+=\{x\in\Eu:x_i\geq 0\text{ for }i=1,\dots,n\}$. Then
  $$ p=P_{\mathbb{R}_+^n}(x)\text{~~where~~}p_i=\max\{0,x_i\}\text{~for all~}i\in\{1,\dots,n\}. $$
\item (Linear systems) Consider a matrix $A\in\mathbb{R}^{m\times n}$  such that $A^TA$ is invertible and suppose {$b\in\operatorname{range}(A)\subseteq\mathbb{R}^m$}. Denote $C=\{x\in\mathbb{R}^n:Ax=b\}$, which is nonempty by assumption. Then
  \begin{equation}\label{eq:linear system proj}
   P_C(x) = x - (A^TA)^{-1}A^T(Ax-b).
  \end{equation}
Compare the computational advantages and disadvantages of using \eqref{eq:linear system proj} with the feasibility problem formulation (for the same problem) outlined in \Cref{example:linear_systems}.
\end{enumerate}
\end{exercise}

\begin{exercise}
Let $C\subseteq\Eu$ be a nonempty, closed set. Prove that the following formulae hold.
\begin{enumerate}
	\item (Translation formula) $P_{y+C}(x)=y+P_C(x-y)$ for all $x,y\in\Eu$.
	\item (Dilution formula) $P_{\alpha C}(x)=\alpha P_C(x/\alpha)$ for all $x\in\Eu$ and $\alpha\in\mathbb{R}\setminus\{0\}$.
\end{enumerate}	
\end{exercise}

\section{Projection algorithms for solving feasibility problems}\label{sec:ProjAlg}

\subsection{The method of cyclic projections}
In this section, we prove convergence of the method of cyclic projections, discussed in \Cref{sec:Preliminaries}, for convex feasibility problems. Given nonempty, closed and convex sets $C_1,\dots,C_r\subseteq\Eu$, we recall that this algorithm is the fixed point iteration defined by
$$x_{k+1} = (P_{C_r}P_{C_{r-1}}\dots P_{C_1})(x_k),\quad\text{for } k=0,1,2,\ldots. $$
In light of \Cref{th:Opial}, the general recipe for proving convergence is as follows: (i) establish that the underlying fixed operator is averaged, (ii) characterize its fixed points in a \emph{meaningful} way, and (iii) apply \Cref{th:Opial} to deduce convergence of the sequence.  To this end, we begin with the following result.

\begin{proposition}\label{prop:cyclic_proj_averaged}
  Let $C_1,\dots,C_r\subseteq\Eu$ be nonempty, closed and convex. Then the operator $P_{C_r}P_{C_{r-1}}\cdots P_{C_1}$ is $\alpha-$averaged with $\alpha=1-2^{-r}$.
\end{proposition}
\begin{proof}
For each $i\in\{1,\dots,r\}$, the projector $P_{C_i}$ is $1/2$-averaged by \Cref{fact:projection_firmly_nonexapnsive}. To prove the result it therefore suffices to establish the following: if $T$ is $1/2$-averaged and $Q$ is $(1-2^{-k})$-averaged for some $k\in\{1,2,\ldots\}$, then $TQ$ is $(1-2^{-(k+1)})$-averaged. To this end, we first write
  $$T=\frac{1}{2}\Id+\frac{1}{2}R\quad\text{and}\quad Q=\frac{1}{2^k}\Id+\left(1-\frac{1}{2^k}\right)S,$$
for nonexpansive mappings $R$ and $S$. Next, observe that
 \begin{align*}
  TQ  =\frac{1}{2}Q+\frac{1}{2}RQ&=\frac{1}{2^{k+1}}\Id+\frac{2^k-1}{2^{k+1}}S+\frac{2^k}{2^{k+1}}RQ\\
     & =\frac{1}{2^{k+1}}\Id+\frac{2^{k+1}-1}{2^{k+1}}\left(\frac{2^k-1}{2^{k+1}-1}S+\frac{2^k}{2^{k+1}-1}RQ\right) \\
     & =\frac{1}{2^{k+1}}\Id+\left(1-\frac{1}{2^{k+1}}\right)\left(\frac{2^k-1}{2^{k+1}-1}S+\frac{2^k}{2^{k+1}-1}RQ\right).
 \end{align*}
 Since $S$ and $RQ$ are nonexpansive, so is their convex combination. It then follows that $TQ$ is $(1-2^{-(k+1)})$-averaged. 
\end{proof}

\begin{proposition}\label{prop:fixed_points_cyclic_proj}
Let $C_1,C_2,\dots,C_r\subseteq\Eu$ be closed and convex sets with nonempty intersection. Then $$\Fix \left(P_{C_r}P_{C_{r-1}}\cdots P_{C_1}\right) = C_1\cap\ldots\cap C_r. $$
\end{proposition}
\begin{proof}
The inclusion $\Fix \left(P_{C_r}P_{C_{r-1}}\cdots P_{C_1}\right) \supseteq \cap_{i=1}^rC_i$ is immediate, so we need only consider the other inclusion. To this end, let $x\in \Fix \left(P_{C_r}P_{C_{r-1}}\cdots P_{C_1}\right)$ and let $c\in\cap_{i=1}^rC_i$. Since $P_{C_i}$ is firmly nonexpansive by \Cref{fact:projection_firmly_nonexapnsive}, we have
\begin{align*}
\|x-c\|^2 &\geq \|(\Id-P_{C_1})x\|^2 + \|P_{C_1}(x)-c\|^2 \\
          &\geq \left(\|(\Id-P_{C_1})x\|^2+\|(\Id-P_{C_2})P_{C_1}(x)\|^2\right)  + \|P_{C_2}P_{C_1}(x)-c\|^2 \\
          &\geq \ddots \\
          &\geq\left( \|(\Id-P_{C_1})x\|^2+\dots+\|(\Id-P_{C_r})(P_{C_{r-1}}\cdots P_{C_1})(x)\|^2 \right) + \|(P_{C_r}P_{C_{r-1}}\cdots P_{C_1})(x)-c\|^2.
\end{align*}
Since $(P_{C_r}P_{C_{r-1}}\cdots P_{C_1})(x)=x$ by assumption, it follows that
 $$ \|(\Id-P_{C_1})x\|=\dots=\|(\Id-P_{C_r})(P_{C_{r-1}}\cdots P_{C_1})(x)\|=0, $$
which implies that $x\in\cap_{i=1}^rC_i$, as claimed.
\end{proof}

\begin{theorem}[Method of cyclic projections]
Let $C_1,C_2,\dots,C_r\subseteq\Eu$ be closed and convex sets with nonempty intersection. Given an initial point $x_0\in\Eu$, consider the iteration generated by
$$x_{k+1} = \left(P_{C_r}P_{C_{r-1}}\cdots P_{C_1}\right)(x_k),\quad\text{for } k=0,1,2,\ldots.$$
Then $(x_k)_{k=0}^\infty$ converges to a point $x^\star\in\cap_{i=1}^rC_i$.
\end{theorem}
\begin{proof}
 The operator $T=P_{C_r}P_{C_{r-1}}\cdots P_{C_1}$ is $(1-2^{-r})$-averaged by \Cref{prop:cyclic_proj_averaged} with $\Fix T=\cap_{i=1}^rC_i\neq\emptyset$ by \Cref{prop:fixed_points_cyclic_proj}. The result immediately follows by applying \Cref{th:Opial}.
\end{proof}

\subsection{The Douglas--Rachford algorithm}\label{sec:DR}

In this section, we introduce another projection algorithm known as the \emph{Douglas--Rachford (DR) algorithm}. It was originally proposed by \cite{DR56} for solving a system of linear equations arising in heat conduction problems, and later extended to solving feasibility problems with arbitrary closed and convex sets by \cite{LM79}. For further details regarding the connection between the original algorithm and the extension of Lions and Mercier, the reader is referred to~\cite[Appendix]{BLM17} as well as \cite{EB92,Svaiter}.  A recent survey on Douglas--Rachford can be found in~\cite{LS18}.

Given two nonempty, closed and convex sets $A,B\subseteq\Eu$, the DR algorithm is the fixed point iteration generated by the \emph{Douglas--Rachford operator}, $T_{A,B}$, defined by
\begin{equation}\label{eq:DRoperator}
x_{k+1}=T_{A,B}(x_k),\quad \text{with } T_{A,B}:=\frac{\Id+R_BR_A}{2}.
\end{equation}
The algorithm iterates by computing an average between the current point and the composition of two reflectors (see~\Cref{fig:DRoperator}). Due to this geometric interpretation, the DR algorithm is sometimes called the \emph{averaged alternating reflections (AAR) method}. In the imaging community, it is also known as the \emph{difference-map algorithm}~(\cite{BCL02}).

\begin{figure}[ht!]
\centering
\includegraphics[width=0.45\textwidth]{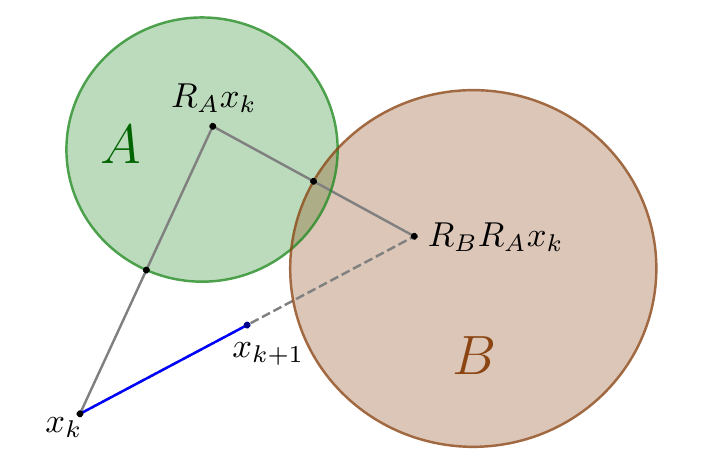}
\caption{Geometric interpretation of the Douglas--Rachford iteration.}\label{fig:DRoperator}
\end{figure}

As in the previous section, we follow the same general recipe with the aim of applying \Cref{th:Opial}. The next result analyzes the nonexpansiveness properties of the DR operator for closed and convex sets.

\begin{proposition}[Nonexpansiveness properties of the DR operator]\label{prop:DRaveraged}
Let $A$ and $B$ be two nonempty, closed and convex subsets of $\Eu$. Then the DR operator $T_{A,B}$ in~\eqref{eq:DRoperator} is $\frac{1}{2}$-averaged and, thus, firmly nonexpansive.
\end{proposition}
\begin{proof}
Since $A$ and $B$ are nonempty, closed and convex, the reflectors $R_A$ and $R_B$ are nonexpansive by~\Cref{fact:projection_firmly_nonexapnsive} and, consequently, so too is their composition. We therefore have that $T_{A,B}=\left(1-\frac{1}{2}\right)\Id+$ $\frac{1}{2}R_BR_A$ with  $R_BR_A$ nonexpansive; i.e., $T_{A,B}$ is $\frac{1}{2}$-averaged. The remaining assertion follows from \Cref{fact:firmly_nonexp}\Cref{fact:firmly_nonexpII}.
\end{proof}

Therefore, according to the previous result and taking into account \Cref{th:Opial}, the convergence of the DR algorithm only depends on the nonemptiness of the set of fixed points of the operator. Furthermore, when they exist, these fixed points need to be useful for solving the feasibility problem. Let us show that this is indeed the case.

\begin{proposition}[Fixed points of the DR operator]\label{prop:DRfixedpoints}
Let $A,B\subseteq\Eu$ be nonempty, closed and convex, and let $T_{A,B}$ be the DR operator. Then
\begin{equation}\label{eq:FixTnonemtpy}
\Fix T_{A,B}\neq\emptyset \quad\iff\quad A\cap B\neq\emptyset,
\end{equation}
and, moreover, $P_A(x)\in A\cap B$ for all $x\in\Fix T$.
\end{proposition}
\begin{proof}
Observe that $A\cap B\subseteq\Fix(R_BR_A)=\Fix T_{A,B}.$ Moreover, $x\in\Fix(R_BR_A)$ if and only if
	\begin{equation*}
	x=(2P_B-\Id)(2P_A-\Id)(x)=2P_B(2P_A(x)-x)-2P_A(x)+x;
	\end{equation*}
that is, $x\in\Fix T_{A,B}$ if and only if $P_B(2 P_A(x)-x)=P_A(x)$.
Consequently, we have $P_A(\Fix T_{A,B})=A\cap B$ and the result follows.
\end{proof}

We are now in position to derive the following theorem which establishes the convergence of the DR algorithm for convex feasibility problems.

\begin{theorem}[Douglas--Rachford method]\label{th:DR}
Let $A,B\subseteq\Eu$ be nonempty, closed and convex sets, and consider the DR operator $T_{A,B}$. Given any initial point $x_0\in\Eu$, consider the iteration generated by
\begin{equation}\label{eq:DRiterationth}
x_{k+1}=T_{A,B}(x_k),\quad\text{for }k=0,1,2,\ldots.
\end{equation}
Exactly one of the following alternatives holds.
\begin{enumerate}
\item $A\cap B\neq\emptyset$ and $(x_k)_{k=0}^{\infty}$ converges to a point $x^\star\in\Fix T_{A,B}$ with $P_A(x^\star)\in A\cap B$.\label{th:DR_I}
\item $A\cap B=\emptyset$ and $\|x_k\|\to+\infty$.\label{th:DR_II}
\end{enumerate}
\end{theorem}
\begin{proof}
The operator $T_{A,B}$ is $\frac{1}{2}$-averaged by~\Cref{prop:DRaveraged}. Suppose first that $A\cap B\neq\emptyset$. Then, according to \Cref{prop:DRfixedpoints}, we have $\Fix T_{A,B}\neq\emptyset$ and therefore we may apply \Cref{th:Opial} to deduce that $$x_k\to x^\star\in\Fix T_{A,B}.$$
Furthermore, \Cref{prop:DRfixedpoints} yields that $P_A(x^\star)\in A\cap B$ and thus \Cref{th:DR_I} is proved. Consider now the case when $A\cap B=\emptyset$. Again, by applying \Cref{prop:DRfixedpoints}, we deduce that $\Fix T_{A,B}=\emptyset$ and hence \Cref{th:DR_II} immediately follows from \Cref{fact:pazy}.
\end{proof}

\begin{remark}[Shadow sequence]\label{rem:DR_shadow}
We emphasize that, in general, the limit point $x^\star$ in~\Cref{th:DR}\Cref{th:DR_I} need not be contained the intersection. Indeed, this is a consequence of the fact that $A\cap B$ can be a strict subset of $\Fix T_{A,B}$. To generate a solution to the feasibility problem using \Cref{th:DR} is therefore necessary to compute $P_A(x^\star)$ (see \Cref{fig:DRscenarios1,fig:DRscenarios2}). In other words, the sequence of interest is not the DR sequence $(x_k)_{k=0}^{\infty}$ itself but rather $(P_A(x_k))_{k=0}^{\infty}$, which is known as the \emph{shadow~sequence}. Also note that, as the projector is a continuous mapping, it immediately follows that
\begin{equation*}
P_A(x_k) \to a^\star=P_A(x^\star)\in A\cap B.
\end{equation*}
\end{remark}

\begin{remark}[Displacement vector]
In the inconsistent case, one typically looks for appropriate surrogate solutions. In the two set case with $A\cap B=\emptyset$, the notion of a \emph{best approximation pair} provides one possibility. A pair of points $(a,b)\in A\times B$ is said to be a \emph{best approximation pair} for $(A,B)$ if
$$\|a-b\|=d(A,B):=\inf\left\{ \|x-y\|: (x,y)\in A\times B\right\}.$$
The finer properties of the DR operator in the potentially inconsistent case have been studied by \cite{BCL04}. Their analysis relies on the notion of \emph{displacement vector}, which is defined as
$v:=P_{\overline{A-B}}(0)$ (i.e., $v$ is the element of minimum norm in the closure of $A-B$).
One can easily check that $v$ measures the gap between the sets, since $\|v\|=d(A,B)$. In addition, $v\in A-B$ if and only if the distance $d(A,B)$ is attained, which means that there exists a best approximation pair. In this case, any inconsistent feasibility problem can be turned into a consistent one described as $A\cap(v+B)$ and the sequence $(x_k)_{k=0}^{\infty}$ generated by the DR algorithm~\eqref{eq:DRiterationth} satisfies
$x_{k+1}-x_k \to v$.
Furthermore, \cite{BM17} showed that when $v\in A-B$, we have
\begin{equation*}
P_A(x_k)  \to a^\star\in A\cap(v+B).
\end{equation*}

Finally, recall that, according to \Cref{th:DR}\Cref{th:DR_II}, the DR sequence is unbounded when the sets do not intersect. Despite this, the latter result asserts that the shadow sequence remains convergent to a point $a^\star$ so long as the distance between the sets, $d(A,B)$, is attained. In this case, $(a^\star, P_B(a^\star))$ would constitute a best approximation pair (see~\Cref{fig:DRscenarios3}).
\end{remark}

\begin{figure}[ht!]
\centering
\subfigure[$A\cap B\neq\emptyset$ and $x^\star\in A\cap B$\label{fig:DRscenarios1}]{\includegraphics[width=0.325\textwidth]{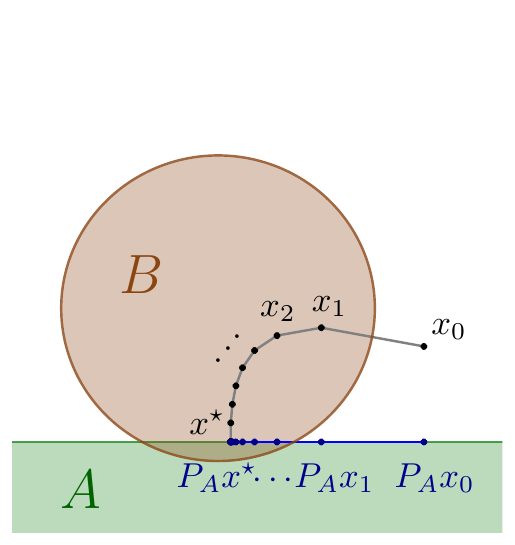}}
\subfigure[$A\cap B\neq\emptyset$ but  $x^\star\not\in A\cap B$\label{fig:DRscenarios2}]{\includegraphics[width=0.325\textwidth]{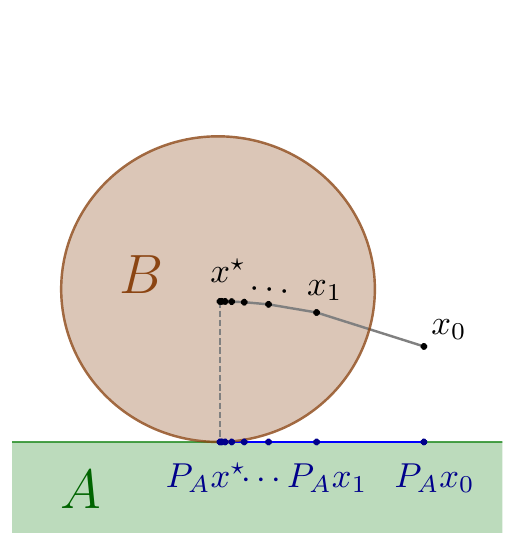}}
\subfigure[$A\cap B=\emptyset$\label{fig:DRscenarios3}]{\includegraphics[width=0.325\textwidth]{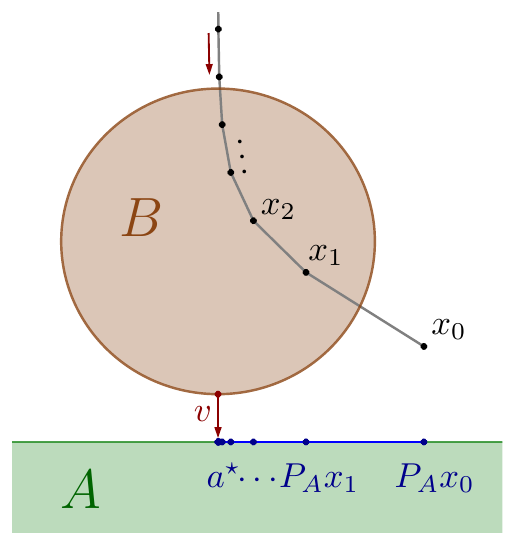}}
\caption{Behavior of the Douglas--Rachford algorithm in three possible scenarios.}\label{fig:DRscenarios}
\end{figure}

\subsection{Product space reformulations and Douglas--Rachford variants}\label{subsec:ProductSpace}
Of the algorithms presented in the previous two sections, only the method of cyclic projections can be applied to the feasibility problem~\eqref{eq:FeasibilityProblem} (without modification) when $r\geq 3$. A naive approach to extending the DR algorithm to the case of three sets $A, B, C\subseteq\Eu$ would be to consider the fixed point iteration generated by
\begin{equation}\label{eq:DRoperator3sets}
T_{A,B,C}:=\frac{\Id+R_CR_BR_A}{2}.
\end{equation}
The corresponding iteration is convergent, since the operator $T_{A,B,C}$ in~\eqref{eq:DRoperator3sets} is firmly nonexpansive with a nonempty set of fixed points, provided that the three sets intersect (see \Cref{exercise:3setDRop}). Unfortunately, the operator's fixed points can not always be used to produce an intersection point, as is shown in \Cref{fig:DRmanysets_fail}.
\begin{figure}[ht]
\centering
\includegraphics[width=0.45\textwidth]{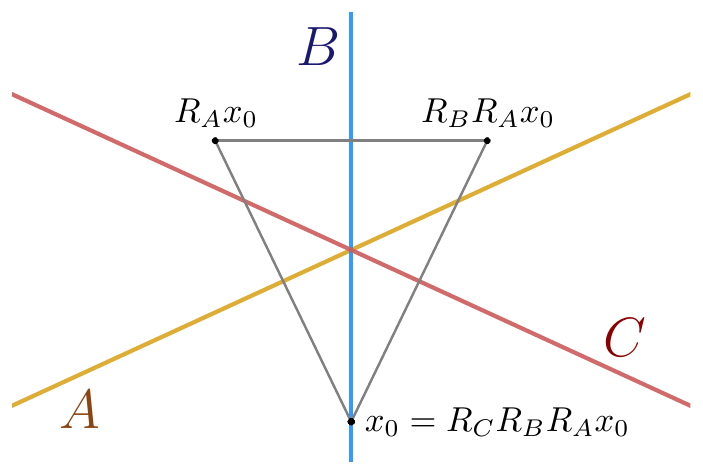}
\caption{Failure of the $3$-sets Douglas--Rachford iteration.}\label{fig:DRmanysets_fail}
\end{figure}

\begin{exercise}\label{exercise:3setDRop}
Let $A,B,C\subseteq\Eu$ be nonempty, closed and convex sets.	
Show that the operator $T_{A,B,C}$, defined in \eqref{eq:DRoperator3sets}, is firmly nonexpansive with $A\cap B\cap C\subseteq\Fix T_{A,B,C}$. Deduce that its fixed point iteration converges whenever $A\cap B\cap C\neq\emptyset$.
\end{exercise}

The classical way to apply the original two-set DR algorithm to problems having $r\geq 3$ sets is to use the following \emph{product space reformulation}, due to \cite{Pierra}. To this end, consider the Euclidean product space
\begin{equation}\label{eq:productHilbertspace}
\mathbf{\Eu}:=\Eu^r=\Eu\times \stackrel{(r)}{\cdots} \times \Eu,
\end{equation}
endowed with the inner product
$$\langle \bs{x},\bs{y}\rangle := \sum_{i=1}^r\langle x_i,y_i\rangle, \quad \text{for all } \bs{x}=(x_i)_{i=1}^r, \bs{y}=(y_i)_{i=1}^r\in\mathbf{\Eu}.$$
Consider the two-set feasibility problem
\begin{equation}\label{eq:prod_fp}
 \text{Find~}\mathbf{x}=(x_1,x_2,\dots,x_r)\in\mathbf{C}\cap\mathbf{D},
\end{equation}
where the constraint sets $\mathbf{C}$ and $\mathbf{D}$ are given by
\begin{equation}\label{eq:CD}
\bs{C}:=C_1\times C_2\times \cdots\times C_r \quad\text{and}\quad \bs{D}:=\{(x,x,\ldots,x)\in\mathbf{\Eu} :x\in\Eu\}.
\end{equation}
Although the set $\bs{D}$, sometimes called the \emph{diagonal} subspace, is always a closed subspace, the properties of $\bs{C}$ are largely inherited from the individual constraint sets. For instance, $\bs{C}$ is nonempty, closed and convex precisely when $C_1,\ldots,C_r$ are. Note that any solution $\mathbf{x}$ to \eqref{eq:prod_fp} satisfies $x_1=x_2=\dots=x_r$.

It is easily seen that the original $r$-set feasibility problem~\eqref{eq:FeasibilityProblem} is equivalent to the two-set problem~\eqref{eq:prod_fp} in the following sense
\begin{equation}\label{eq:PS_equiv_feasibility}
x\in\bigcap_{i=1}^r C_i\subseteq \Eu \quad \iff \quad \bs{x}=(x,x,\ldots,x)\in \bs{C}\cap\bs{D}\subseteq\mathbf{\Eu}.
\end{equation}
In other words, an $r$-set feasibility problem in $\Eu$ can always be formulated as a two-set feasibility problem in $\mathbf{\Eu}$. It is worth mentioning that when the number of constraints, $r$, is large, this formulation can become numerically inefficient because it requires computations to be performed in a space with much higher dimension. In some special cases, this can be avoided (see \Cref{exercise:averaged_projections}).

In order to apply projection methods to the reformulated problem, it is necessary to have access to the projectors onto the sets $\bs{C}$ and $\bs{D}$. As we summarize in the following proposition, this is indeed the case whenever the projectors onto the underlying constraint sets in the original problem, $C_1,\ldots,C_r$, are available.
\begin{proposition}[Product-space projectors]\label{prop:Product_Space}
Let $\bs{x}=(x_1,\ldots,x_r)\in\bs{\Eu}$. The projectors onto the sets $\bs{C}$ and $\bs{D}$ in \eqref{eq:CD} at $\bs{x}$ are given by
  $$ P_{\bs{C}}(\bs{x})=P_{C_1}(x_1)\times P_{C_2}(x_2)\times\dots\times P_{C_r}(x_r)\quad\text{and}\quad P_{\bs{D}}(\bs{x}) = \left(\frac{1}{r}\sum_{i=1}^rx_i, \frac{1}{r}\sum_{i=1}^rx_i, \ldots, \frac{1}{r}\sum_{i=1}^rx_i\right). $$
\end{proposition}
\begin{proof}
The proof is left as an exercise.  
\end{proof}

\begin{exercise}
Prove \Cref{prop:Product_Space}.
\emph{Hint:}~To prove the formula for $P_C$, use the definition directly. To prove the formula for $P_D$, note that $D$ is a subspace and use \Cref{fact:projection}.	
\end{exercise}	

\begin{exercise}[The method of averaged projections]\label{exercise:averaged_projections}
Suppose $C_1, C_2, \dots, C_r\subseteq\Eu$ are closed and convex sets with nonempty intersection. Given a point $x_0\in\Eu$, the \emph{method of averaged projections} is the fixed point iteration governed by
  $$ x_{k+1}=T(x_k),\text{~~where~~}T=\frac{P_{C_r}+\cdots + P_{C_1}}{r}. $$
By applying the method of cyclic projections to \eqref{eq:prod_fp}, show that $(x_k)_{k=0}^\infty$ converges to a point in $\cap_{i=1}^rC_i$.
\end{exercise}

The product space reformulation explained above leads to the following DR scheme that can be applied to $r$ closed and convex sets $C_1,C_2,\ldots,C_r\subseteq\Eu$.

\paragraph{Douglas--Rachford in the product space.}
Given $r$ arbitrary starting points $x_{1,0},x_{2,0},\ldots,x_{r,0}\in\Eu$, the iterative scheme can be expressed as
\begin{equation}\label{eq:DR_productspace}
\begin{aligned}
&\text{for } k=0, 1, 2, \ldots:\\
&\left\lfloor\begin{array}{l}
p_k=\frac{1}{r}\sum_{i=1}^r x_{i,k},\\
\text{for } i=1, 2,\ldots, r:\\
\left\lfloor\begin{array}{l}
x_{i,k+1}=\frac{1}{2}x_{i,k}+\frac{1}{2}R_{C_i}(2p_k -x_{i,k}).
\end{array}\right.
\end{array}\right.
\end{aligned}
\end{equation}
Observe that \eqref{eq:DR_productspace} is just the iteration
\begin{equation*}
\bs{x}_{k+1}=\frac{1}{2}\bs{x}_k+\frac{1}{2}R_{\bs{C}}R_{\bs{D}}(\bs{x}_k),\quad\text{ for } k=0,1,2,\ldots,
\end{equation*}
where $\bs{C}$ and $\bs{D}$ are the product and the diagonal sets defined in~\eqref{eq:CD}, and $\bs{x}_k:=(x_{1,k},x_{2,k},\ldots,x_{r,k})$.
Consequently, \Cref{th:DR} guarantees convergence of the sequences generated by \eqref{eq:DR_productspace}.

We remark that the choice to compute the reflection with respect to the diagonal $\bs{D}$ first is deliberate, so that the shadow sequence, $(P_{\bs{D}}(x_k))_{k=0}^{\infty}\subseteq\mathbf{\Eu}$, can be unambiguously identified with the sequence $(p_k)_{k=0}^{\infty}$ in $\Eu$ (rather than $\mathbf{\Eu}$). The latter will converge to a common point of the sets, whenever such a point exists (see~\Cref{fig:DRproductspace}).

\begin{figure}[ht!]
\centering
\includegraphics[width=0.45\textwidth]{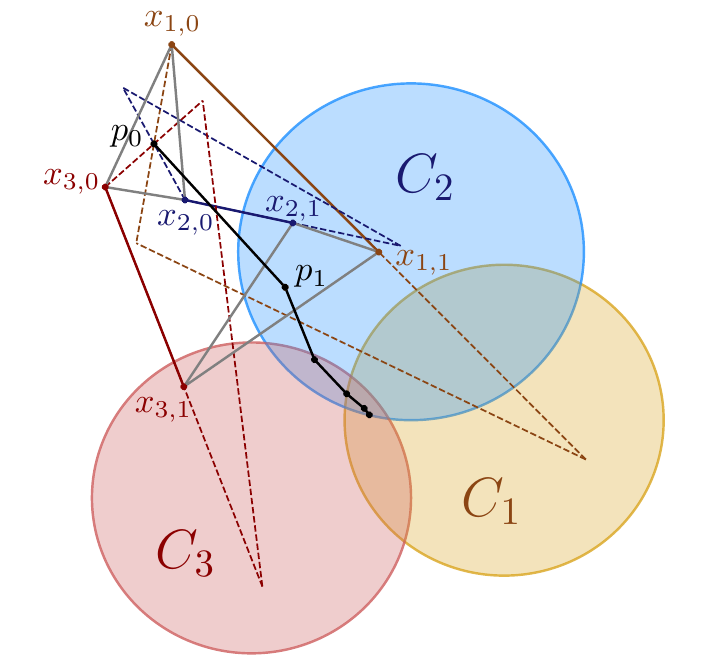}
\caption{Illustration of the Douglas--Rachford iteration in the product space.}\label{fig:DRproductspace}
\end{figure}

Note that the iteration in \eqref{eq:DR_productspace} has the advantage that each of the individual projectors, $P_{C_i}$, can be computed in parallel. On the other hand,  as many points as  constraint sets need to be stored at each step to compute the next iteration. This makes the algorithm computationally inefficient when the number of constraints is large.

One possible modification of the DR algorithm which avoids using a product reformulation is based on cyclic application of DR operators.

\subsection{Cyclic Douglas--Rachford methods}
Instead of the operator in~\eqref{eq:DRoperator3sets},  \cite{BT14} introduced the \emph{cyclic DR operator} defined as
\begin{equation}\label{eq:DRoperator_cyclic}
T_{[C_1,C_2,\ldots,C_r]}:=T_{C_r,C_1}T_{C_{r-1},C_r}\cdots T_{C_2,C_3}T_{C_1,C_2}.
\end{equation}
The fixed point iteration generated by this operator, known as the \emph{cyclic DR method}, can be applied to convex feasibility problems defined by an arbitrary number of sets, without recourse to the product space reformulation \cite[see][Theorem~3.2]{BT14}. Note the operator \eqref{eq:DRoperator_cyclic} can be viewed as
cyclically applying the classical DR method to pairs of sets (see~\Cref{fig:DRcyclic}). The analysis of the method in the inconsistent case was developed in~\cite{BT15}, and a further extension of the algorithm has recently been proposed by~\cite{CG18}, which incorporates $k$-sets-DR operators
(of type~\eqref{eq:DRoperator3sets} when $k=3$).
Observe that
$$T_{[A,B]}=T_{B,A}T_{A,B} \neq T_{A,B},$$
that is, \eqref{eq:DRoperator_cyclic} does not coincide with the classical DR operator for $r=2$. Alternatively, \cite{anchored} proposed the \emph{cyclically anchored DR method}, which is defined by the operator
\begin{equation}\label{eq:DRoperator_anchored}
T_{C_1,[C_2,\ldots,C_r]}:=T_{C_1,C_r}T_{C_1,C_{r-1}}\cdots T_{C_1,C_2}.
\end{equation}
Here the set $C_1$ is called the \emph{anchor set}.
In contrast to the cyclic DR method, the latter scheme does reduce to the original DR  scheme when dealing with only two~sets. Also, when $C_1=\Eu$, observe that the cyclically anchored DR method coincides with the method of cyclic projections, since $T_{\Eu,C_i}=P_{C_i}$.

\begin{figure}[ht!]
\centering
\subfigure[Cyclic Douglas--Rachford method\label{fig:DRcyclic}]{\includegraphics[width=0.45\textwidth]{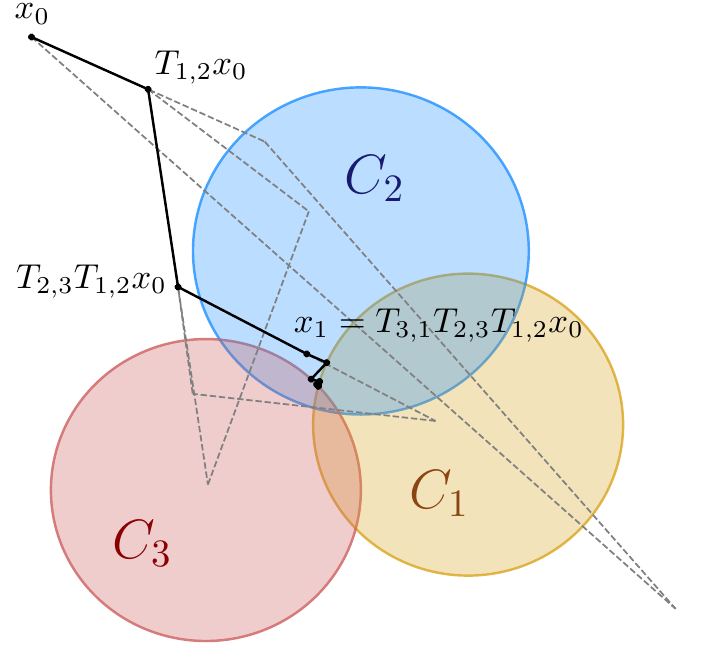}}\hspace{0.05\textwidth}\subfigure[Cyclically anchored Douglas--Rachford method\label{fig:DRanchored}]{\includegraphics[width=0.45\textwidth]{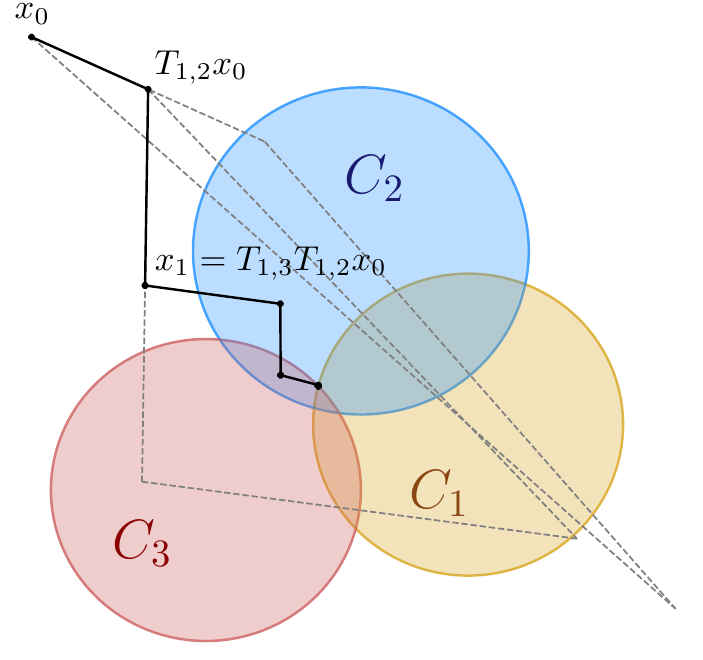}}
\caption{Illustration of two cyclic versions of Douglas--Rachford for finitely many sets.}\label{fig:DRmanysets}
\end{figure}

\subsection{Some modified and relaxed versions}
To conclude this section, we briefly mention some other modifications and relaxations of the classical DR algorithm in the literature. Such variants are typically constructed by relaxing or modifying some of the involved operators, so that parameters can be incorporated into the scheme. The introduction of these parameters, which can be tuned, allows the possibility of accelerating the method or even completely changing the dynamics of the algorithm.

In what follows, we present four such variants. Each of their iterations are illustrated in~\Cref{fig:DRrelaxed}.

\paragraph{Generalized Douglas--Rachford.} The \emph{generalized Douglas--Rachford (GDR) method} is the most evident generalization, which was in fact the method studied by \cite{EB92}. The algorithm relies on iteratively applying a general $\alpha$-averaged version of the DR operator, that is,
\begin{equation}\label{eq:GDR}
x_{k+1}=T_{A,B,\alpha}(x_k):=\left(1-{\alpha}\right)x_k+{\alpha}R_BR_A(x_k), \quad \text{with } \alpha\in{]}0,1{[}.
\end{equation}
Note that for $\alpha=\frac{1}{2}$ it becomes the classical DR algorithm.

\paragraph{Relaxed Averaged Alternating Reflections.} \cite{L08} proposed the \emph{relaxed averaged alternating reflections (RAAR) method}, whose iteration is defined as an average between the classical DR and the projection onto the first set, that is,
\begin{equation*}
x_{k+1}=\left(1-\beta\right)P_A(x_k)+\beta T_{A,B}(x_k), \quad \text{with } \beta\in{]}0,1{[}.
\end{equation*}
Unlike the classical DR method, this scheme converges even for inconsistent feasibility problems (without turning to the shadow sequence) as long as the distance between the sets is attained. In this case, the method provides a best approximation pair. It is worth mentioning that a more general relaxed version of DR has been recently proposed by \cite{Thao2018}.

\paragraph{Circumcentered Douglas--Rachford.} Motivated by the ``{spiraling dynamics}'' exhibited by DR when applied to subspaces (see~\Cref{fig:DRsubs1}), \cite{BBS18} proposed a modification of the algorithm termed the \emph{circumcentered Douglas--Rachford (CDR) method}. In this method, the next iterate is computed as the circumcenter of the triangle implicitly generated by the classical DR iteration; that~is,
\begin{equation}\label{eq:CDR}
x_{k+1}=C_T[x_k,R_A(x_k),R_BR_A(x_k)],
\end{equation}
where $C_T[a,b,c]$ denotes the circumcenter of the triangle of vertices $a$, $b$ and $c$.
This method requires the sets to be closed subspaces since, otherwise the circumcenter operator is not
 necessarily well-defined. In fact, even when it is, the method may still fail to converge. An example
with two intersecting balls in $\mathbb{R}^2$ is shown in \Cref{fig:circ_balls}. For one-dimensional subspaces, however, the algorithm always converges in one iteration.

\paragraph{Averaged Alternating Modified Reflections.} \label{sec:AAMR} In the particular case when the DR algorithm is applied to closed (affine) subspaces, it does not yield just any point in the intersection but rather the intersection nearest to the initial one, i.e., $P_{A\cap B}(x_0)$ \cite[see, e.g.,][Corollaries~4.4 and 4.5]{BCNPW14}. In this context, the method can be used to solve best approximation problems. This is not the case for arbitrary convex sets, as shown in \Cref{fig:DRsubspaces}.
\cite{AAMR} proposed the \emph{averaged alternating modified reflections (AAMR) method} to solve best approximation problems described by arbitrary closed and convex sets. This algorithm can be seen as a variant of DR where the reflection steps are slightly modified. For a given point $q\in\Eu$, the iteration is given by
\begin{equation}\label{eq:AAMR}
x_{k+1}=\left(1-\alpha\right)x_k+\alpha(2\beta P_{B-q}-\Id)(2\beta P_{A-q}-\Id)(x_k), \quad \text{with } \alpha\in{]}0,1{]},\beta\in{]}0,1{[}.
\end{equation}
Under a constraint qualification on the sets at the point $q$, the scheme converges to a point $x^\star$ such that $P_A(x^\star+q)=P_{A\cap B}(q)$. The convergence of the sequence $(x_k)_{k=0}^\infty$ for the case $\alpha=1$ was proved by \cite{AAMR_asymptotic}. Observe that the DR algorithm can be obtained as the limit case $\beta=1$ in~\eqref{eq:AAMR} with $\alpha=\frac{1}{2}$ and $q=0$.

\begin{figure}[ht!]
\centering
\vspace{0.8cm}
\subfigure[Both sets $A$ and $B$ are closed subspaces. The DR algorithm converges to $x^\star=P_{A\cap B}(x_0)$\label{fig:DRsubs1}]{\includegraphics[width=0.45\textwidth]{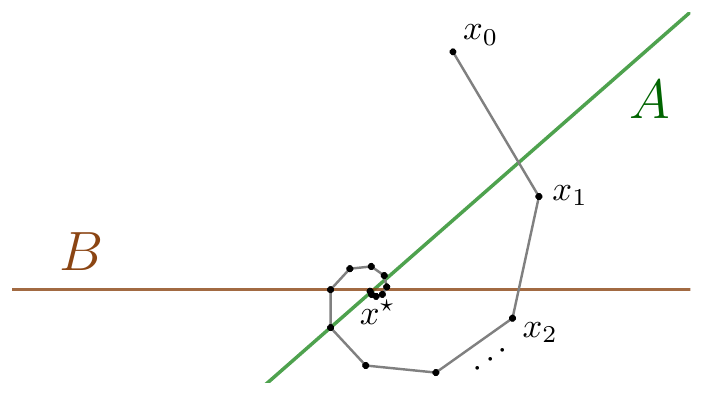}}\hspace{0.05\textwidth}
\subfigure[$B$ is a closed subspace and $A$ is a halfspace. The DR algorithm converges to some point $x^\star\in A\cap B$]{{\includegraphics[width=0.45\textwidth]{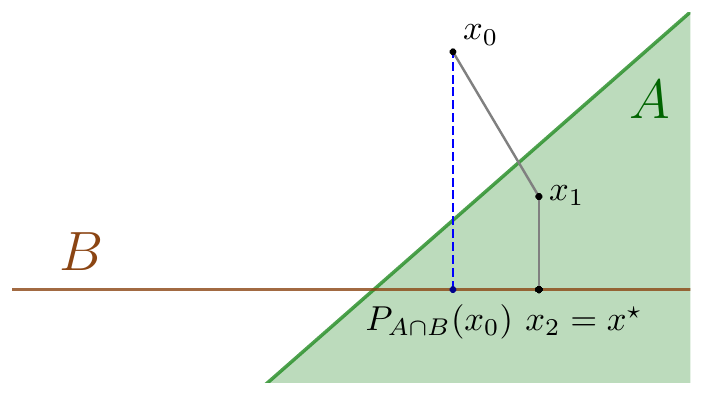}}}
\caption{Failure of the DR method for solving the best approximation problem for arbitrary convex sets.}\label{fig:DRsubspaces}
\end{figure}

\begin{figure}[ht!]
\centering
\subfigure[GDR with $\alpha=0.8$\label{fig:GDR}]{\includegraphics[width=0.4\textwidth]{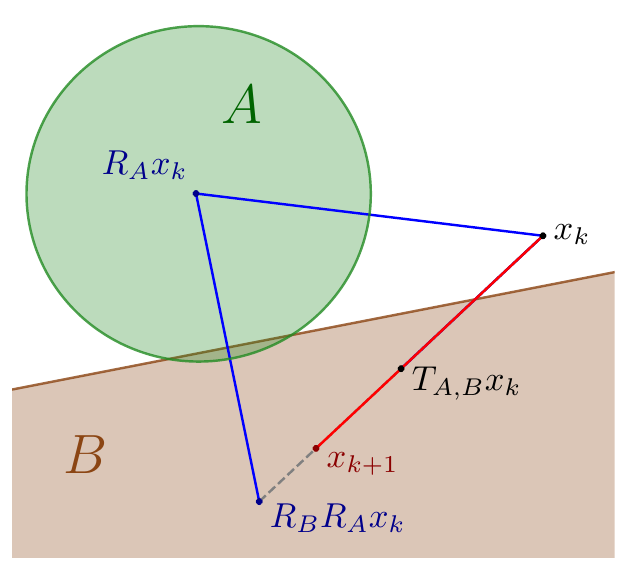}}\hspace{0.1\textwidth}%
\subfigure[RAAR with $\beta=0.4$\label{fig:RAAR}]{\includegraphics[width=0.4\textwidth]{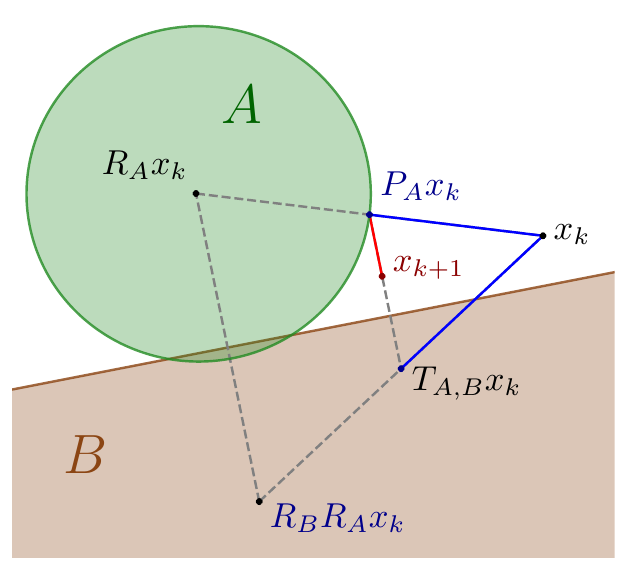}}\\[0.5cm]

\subfigure[CDR\label{fig:CDR}]{\includegraphics[width=0.4\textwidth]{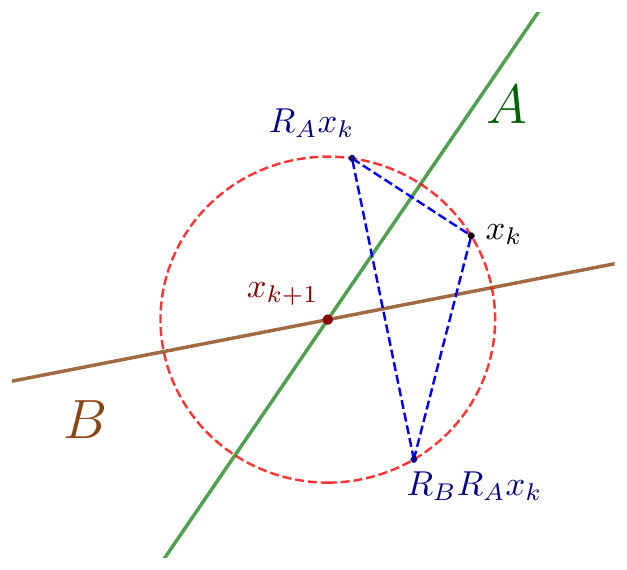}}\hspace{0.1\textwidth}%
\subfigure[AAMR with $\alpha=0.5$, $\beta=0.8$ and $q=(0,0)$\label{fig:AAMR}]{\includegraphics[width=0.4\textwidth]{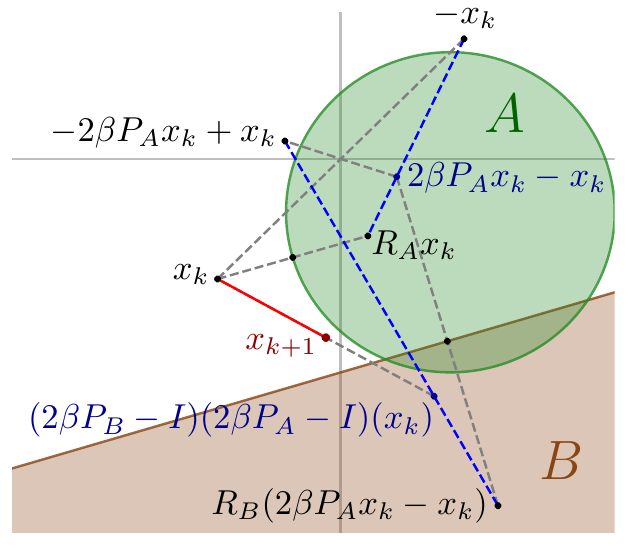}}
\caption{One iteration of GDR, RAAR, CDR and AAMR methods.}\label{fig:DRrelaxed}
\end{figure}

\begin{figure}[ht!]
\centering
\vspace{1.4cm}
\includegraphics[width=0.3\textwidth]{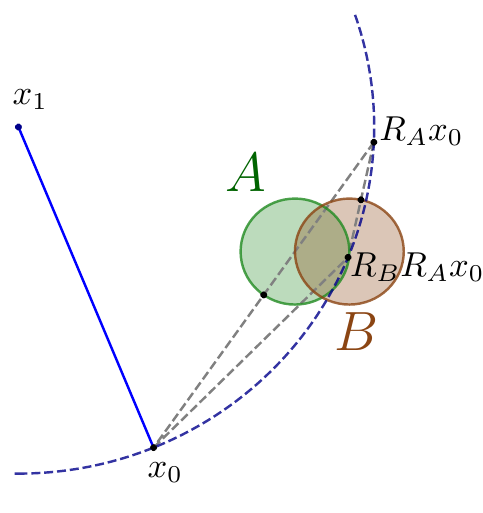}
\includegraphics[width=0.45\textwidth]{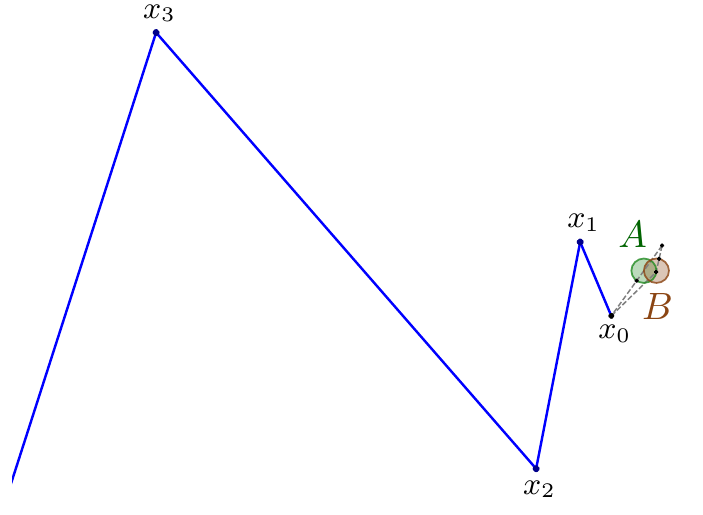}
\caption{Failure of the circumcentered Douglas--Rachford method when applied to two balls of the same radius $A,B\in\R^2$: the sequence $(x_k)_{k=0}^\infty$ diverges. The left figure shows construction of the first iteration from the right figure in more detail.}\label{fig:circ_balls}
\end{figure}

\section{A successful heuristic for nonconvex problems}\label{sec:Nonconvex}

The DR algorithm has recently gained significant popularity, in part, thanks to its good behavior when applied in nonconvex settings. In this context, we note that the DR operator may be multivalued due to the fact that the projection onto a nonconvex set is not necessarily unique. Consequently, the equality in~\eqref{eq:DRiterationth} must be replaced by an inclusion, and the iteration thus takes the form
\begin{equation}\label{eq:NonconvexDR}
x_{k+1}\in T_{A,B}(x_k):=\left\{x_k+b_k-a_k\in \Eu: a_k\in P_A(x_k), b_k\in P_B(2a_k-x_k)\right\}.
\end{equation}
The fixed point set of this operator remains useful for solving the feasibility problem without assuming convexity of the sets $A$ and $B$. This is explained in the following straightforward exercise which  generalizes \Cref{prop:DRfixedpoints}.
\begin{exercise}[Fixed points of the DR operator]\label{exercise:DRfixedpoints-set}
	Let $A,B\subseteq\Eu$ be nonempty and closed, and let $T_{A,B}$ be the DR operator given by \eqref{eq:NonconvexDR}. Prove that
	\begin{equation*}
		\Fix T_{A,B}\neq\emptyset \quad\iff\quad A\cap B\neq\emptyset,
	\end{equation*}
	and that, for any $x\in\Fix T_{A,B}$, there exists a point $p\in P_A(x)$ such that $p\in A\cap B$.
\end{exercise}

Despite the fact that convergence of the algorithm is generally only ensured when applied to convex sets, the method has nevertheless been successfully employed for solving a variety of nonconvex optimization problems, especially those of combinatorial nature. Examples of such applications include matrix completion~\citep{ABTmatrix}, protein conformation determination~\citep{BT17}, phase and bit retrieval~\citep{BCL02,ElserPhase,ElserBit}, differential equations~\citep{LLS17}, graph coloring~\citep{Graphcol,Graphcol2}, combinatorial designs~\citep{Designs}, and a wide variety of NP-hard problems such as Sudoku or 3-Satisfiability~\citep{ABTcomb,Elser,schaad2010modeling}.

The nonconvex theory for the DR algorithm is much less developed; there are very few results explaining why the algorithm works, and even less justifying its good global performance. The first nonconvex scenario was considered by \cite{BS11}, who established local convergence of the method near each of the intersection points of a line and a sphere in a Euclidean space. An explicit description of the regions of convergence was later provided by \cite{ABglobal}. It was finally \cite{benoist} who, via the construction of a Lyapunov function, established the convergence of the algorithm for every starting point not lying on the hyperplane of symmetry. Lyapunov functions are a powerful tool from difference inclusions whose existence guarantees the convergence of the iteration. By using this approach, \cite{DT18} proved global convergence of the DR algorithm for finding a zero of a function, with applications to several nonconvex feasibility problems.

From a different perspective, \cite{ABT16} proved global convergence for the case of a halfspace and a finite set. We make use of this scenario to illustrate in~\Cref{fig:DR_AP_nonconvex} the difference between the behavior of DR and cyclic projections when addressing combinatorial problems. While the method of cyclic projections usually gets stuck in those points which are close to be solutions (as it finds a local best approximation pair), DR is satisfactorily capable to escape from them (thanks to~\Cref{th:DR}\Cref{th:DR_II}). Even so, the DR algorithm does not break free from getting caught by cycles in other nonconvex settings. Although this is something that does not seem to happen very often, it may be hard to detect. It is worth mentioning that the cycling of the algorithm for a simple inconsistent nonconvex feasibility problem, specifically, a hyperplane and a doubleton, was recently analyzed by \cite{BDLdoubleton}.

\begin{figure}[ht!]
\centering
\subfigure[Cyclic projections\label{fig:APnonconvex}]{\includegraphics[width=0.35\textwidth]{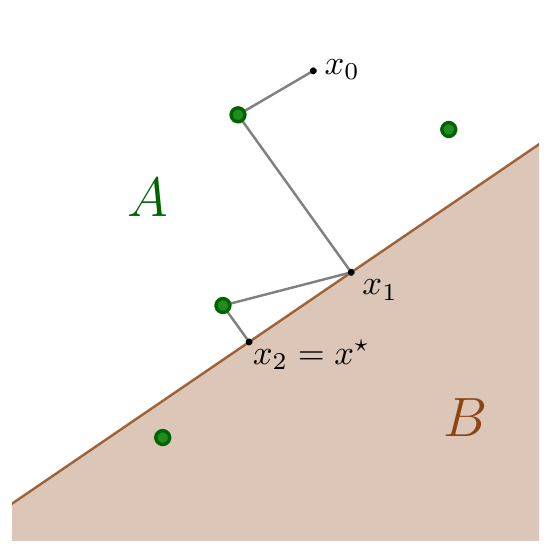}}
\hspace{0.1\textwidth}
\subfigure[Douglas--Rachford\label{fig:DRnonconvex}]{\includegraphics[width=0.35\textwidth]{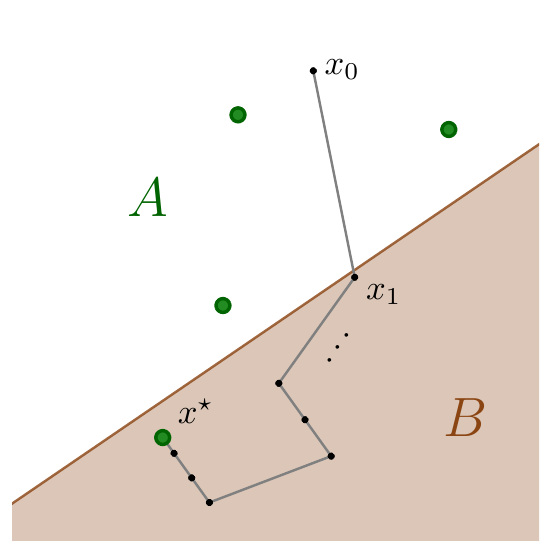}}
\caption{The method of cyclic projections and DR applied to a finite set (the set of green points) and a halfspace.}\label{fig:DR_AP_nonconvex}
\end{figure}

Local convergence of the algorithm (i.e., convergence for starting points sufficiently close to a solution) in nonconvex settings has been established, for instance, for the case of a line and an ellipse or a $p$-sphere~\citep{BLSSS18}, and for union of convex sets~\citep{BNlocal,tam2018algorithms,dao2019union}. Other results regarding local convergence are usually obtained by requiring regularity properties of the sets and/or of their intersection, see e.g.~\cite{BD17,HLnonconvex,HLNsparse,Plinear}.%

\subsection{Application to a generalized ``8-queens problem''}\label{sec:queens}
In this section, we use the DR algorithm to solve a generalization of the so-called \emph{$8$-queens problem}, which asks for the placement of $8$ mutually non-attacking queens on an $8\times 8$ chessboard. Recall that, in chess, queens can move any number of squares along a vertical, horizontal or diagonal of the board. For this reason, two or more queens are said to be \emph{attacking} if they lie in the same vertical, horizontal or diagonal.

The $8$-queens puzzle therefore consists in placing $8$ queens so that each row and column contain precisely one queen, and each diagonal contains at most one. The problem we consider here generalizes this puzzle in two ways. Firstly, we consider $n\times n$ chessboards for an arbitrary positive integer $n$, and secondly, we consider placing $mn$ queens ($1\leq m\leq n$) so that there are precisely $m$ queens in each row and column, and each diagonal contains at most $m$. We refer to the general problem as the \emph{$(m,n)$-queens problem}. For the sake of clarity, we shall restrict our exposition to the $m=2$ case, however the results easily generalize to the $m>2$ case. An example of a solution to the $(2,8)$-queens problem is shown in~\Cref{fig:8Q}.

For the case when $m=1$, the $(1,n)$-queens problem is known as the \emph{$n$-queens problem}.
In his master's thesis, \cite{schaad2010modeling} successfully demonstrated that the DR method can be used to solve the $n$-queens problem.

\begin{figure}[htb]
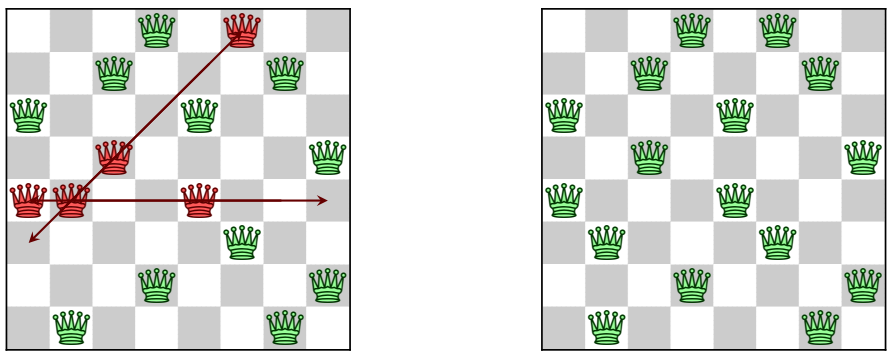

	\centering
	{\setchessboard{showmover=false}
		\newgame
		\chessboard[%
		maxfield=h8,%
		boardfontencoding=LSBC3,
		boardfontsize=16pt,%
		labelfontsize=6pt,%
		whitefieldcolor=white,%
		whitefieldmaskcolor=white,%
		blackfieldmaskcolor=black!20!white,%
		blackfieldcolor=black!20!white,%
		whitepiececolor=green!20!black,%
		whitepiecemaskcolor=green!40!white,%
		blackpiececolor=red!70!white,%
		blackpiecemaskcolor=red!70!white,%
		setfontcolors,%
		label=false,%
		margin=false,%
	    setwhite={Qb4, Qe4, Qc5, Qf8, Qa4, Qa6, Qb1, Qc7, Qd8, Qd2, Qe6, Qf3, Qg7, Qg1, Qh5, Qh2},%
        coloremph,
        whitepiecemaskcolor=red!65!white,
        whitepiececolor=black!65!red,
        empharea=a4-e5,
        empharea=f8-f8,
	    pgfstyle=straightmove,%
	    arrow=stealth,%
	    linewidth=.1ex,%
	    padding=0ex,%
	    color=red!40!black,%
	    pgfstyle=straightmove,%
	    shortenstart=2ex,%
	    markmoves={a3-f8,f8-a3,h4-a4,a4-h4}
	    ]}\hspace{0.15\textwidth}
	    {\setchessboard{showmover=false}
		\newgame
		\chessboard[%
		maxfield=h8,%
		boardfontencoding=LSBC3,
		boardfontsize=16pt,%
		labelfontsize=6pt,%
		whitepiececolor=green!20!black,%
		whitepiecemaskcolor=green!40!white,%
		whitefieldmaskcolor=white,%
		blackfieldmaskcolor=black!20!white,%
		blackfieldcolor=black!20!white,%
		setfontcolors,%
		label=false,%
		margin=false,%
	    setwhite={Qa6, Qa4, Qb3, Qb1, Qc7, Qc5, Qd8, Qd2, Qe6, Qe4, Qf8, Qf3, Qg7, Qg1, Qh5, Qh2},%
	    pgfstyle=straightmove,%
	    arrow=stealth,%
	    linewidth=.1ex,%
	    padding=0ex,%
	    color=green!40!black,%
	    pgfstyle=straightmove,%
	    shortenstart=2ex
	    ]}

\caption{Two  configurations of 16 queens on an $8\times 8$ chessboard. The left board is not a solution to the $(2,8)$-queens problem as one row and one backward diagonal both have three queens (red). In contrast, the right board is a solution.}\label{fig:8Q}
\end{figure}

To model queens on a chessboard as a feasibility problem, consider a matrix $X\in \Bnn$. We interpret the entries of this matrix as squares on a chessboard, with non-zero entries corresponding to queens. The $(2,n)$-queens problem can thus be described in terms of the following constraints:

\begin{itemize}
\item[1.] Exactly two entries in each row of $X$ are non-zero.
\item[2.] Exactly two entries in each column of $X$ are non-zero.
\item[3.] At most two entries in each forward diagonal of $X$ are non-zero.
\item[4.] At most two entries in each backward diagonal of $X$ are non-zero.
\end{itemize}

We now express the problem mathematically using these constraints. To this end, observe that a matrix $X\in\Rnn$ is a solution to the $(2,n)$-queens problem if and only if its entries are contained in $\bin$ and it belongs to the following constraint sets

\begingroup
\allowdisplaybreaks
\begin{subequations}\label{eq:n3l_const}
 	\begin{align}
 	C_1  &:= \left\{ X\in \Rnn :  \sum_{i=1}^nx_{ij}= 2,\,\text{for all } j=1,\ldots,n.\right\}, \label{eq:C1} \\
 	C_2  &:= \left\{ X\in \Rnn  : \sum_{j=1}^nx_{ij}= 2,\,\text{for all } i=1,\ldots,n. \right\}, \label{eq:C2} \\
 	C_3  &:= \left\{X\in \Rnn  : \sum_{i=1}^{n-k}x_{i,i+k}\leq 2 \text{ and } \sum_{j=1}^{n-k}x_{j+k,j}\leq 2,\,\text{for all } k=0,1,\ldots,n-3.\right\}, \label{eq:C3} \\
 	C_4  &:= \left\{ X\in \Rnn  :  \sum_{j=1}^{n-k}x_{n-j-k,j}\leq 2 \text{ and }  \sum_{i=1}^{n-k}x_{n-i,i+k}\leq 2,\,\text{for all } k=0,1,\ldots,n-3.\right\}. \label{eq:C4}	
 	\end{align}
\end{subequations}
\endgroup
Now, even with these constraints sets fixed, there are still many possible formulations which model the desired feasible set.  Let us denote $\widehat{C}_i:=C_i\cap \Bnn$ for each $i=1,2,3,4$. We shall concern ourselves with the following four formulations:

\begin{subequations}
\begin{align}
   \text{Find } X & \in {C}_1\cap {C}_2\cap {C}_3\cap {C}_4\cap \Bnn \tag{Formulation 1}\label{For1}\\
    & = {C}_1\cap {C}_2\cap \widehat{C}_3\cap \widehat{C}_4 \tag{Formulation 2}\label{For2}\\
    & = \widehat{C}_1\cap \widehat{C}_2\cap {C}_3\cap {C}_4 \tag{Formulation 3}\label{For3}\\
   & = \widehat{C}_1\cap \widehat{C}_2\cap \widehat{C}_3\cap \widehat{C}_4. \tag{Formulation 4}\label{For4}
\end{align}
\end{subequations}

In order to apply any of the algorithms previously presented, we need to be able to compute the projections onto each of these sets. Note that a projection of any $X\in\Rnn$ onto $\Bnn$, $\pi_{\Bnn}(X)\in P_{\Bnn}(X)$, can be easily computed component-wise as
\begin{equation*}
\left(\pi_{\Bnn}(X)\right)_{ij}=\left\{\begin{array}{ll}
1, & \text{if } x_{ij}>0.5 ,\\
0, & \text{otherwise}, \end{array}\right.\quad\text{for } i,j=1,2,\ldots,n.
\end{equation*}
In order to derive explicit formulae to compute projections onto the remaining sets, consider, for each $p=1,2,\ldots,n$, the sets

\begin{gather*}
S_{p}:=\left\{x\in \R^p : \sum_{i=1}^px_i= 2\right\} \quad \text{ and } \quad H_{p}:=\left\{x\in \R^p : \sum_{i=1}^px_i\leq 2\right\},
\end{gather*}
where we use the notation $x=(x_1,x_2,\ldots,x_p)$. Similarly as before, let us denote $\widehat{S}_p:=S_p\cap \bin^p$ and $\widehat{H}_p:=H_p\cap \bin^p$. Although $S_p$ and $H_p$ are, respectively, a hyperplane and a halfspace and thus, convex, the sets $\widehat{S}_p$ and $\widehat{B}_p$ are discrete. The constraint sets in~\eqref{eq:n3l_const}, as well as their discrete counterparts, can be represented in terms of these four sets. The following propositions provide specific formulae for computing a projection onto each of these sets.

\begin{proposition}[Projectors onto $S_p$ and $H_p$]\label{prop:projH}
Let $x=(x_1,x_2,\ldots,x_p)\in\R^p$. Then the projectors onto the sets $S_{p}$ and $H_p$ are given component-wise by
\begin{align*}
\left(P_{S_{p}}(x)\right)_i&=x_i+\frac{1}{p}\left(2-\sum_{j=1}^p x_j\right) , \quad \text{for } i=1,2,\ldots,p;\\
\left(P_{H_{p}}(x)\right)_i&=x_i+\frac{1}{p}\min\left\{0,2-\sum_{j=1}^p x_j\right\}, \quad \text{for } i=1,2,\ldots,p.
\end{align*}
\end{proposition}
\begin{proof}
Observe that  $S_{p}=\left\{x\in\R^p : e^T x =  2\right\}$ and $H_{p}=\left\{x\in\R^p : e^T x\leq 2\right\}$ with $e:=(1,1,\ldots,1)\in\R^p$. Thus, the result follows directly from~\eqref{eq:proj_hyperplane} and \eqref{eq:proj_halfspace}.
\end{proof}

\begin{proposition}[Projections onto $\widehat{S}_p$ and $\widehat{H}_p$]\label{prop:projB}
Let $x=(x_1,x_2,\ldots,x_p)\in\R^p$. Then a projection onto the sets $\widehat{S}_p$ and $\widehat{H}_p$, $\pi_{\widehat{S}_p}(x)\in P_{\widehat{S}_p}(x)$ and $\pi_{\widehat{H}_p}(x)\in P_{\widehat{H}_p}(x)$, can be computed component-wise as
\begin{align}
\left(\pi_{\widehat{S}_p}(x)\right)_i &=\left\{\begin{array}{ll}
1, & \text{if } i\in Q_2(x),\\
0, & \text{otherwise}, \end{array}\right.\quad\text{for } i=1,2,\ldots,p, \label{eq:piBS}\\
\left(\pi_{\widehat{H}_p}(x)\right)_i &=\left\{\begin{array}{ll}
1, & \text{if } i\in Q_2(x) \text{ and } x_i>0.5 ,\\
0, & \text{otherwise}, \end{array}\right.\quad\text{for } i=1,2,\ldots,p;\label{eq:piBH}
\end{align}
where $Q_2(x)$ is the set of indices in $\{1,2,\ldots,p\}$ corresponding to the two largest values in $\{x_1,x_2,\ldots,x_p\}$ (largest index is chosen in case of tie).
\end{proposition}
\begin{proof}
Let $e_1,\ldots,e_p$ denote the unit vectors of the standard basis of $\R^p$ and note that
\begin{align*}
\widehat{S}_p&=\{e_i+e_j : i,j=1,2,\ldots,p,\, i\neq j\},\\
\widehat{H}_p&=\{0_p\}\cup \{e_i : i=1,2,\ldots,p\} \cup \{e_i+e_j : i,j=1,2,\ldots,p,\, i\neq j\}.
\end{align*}
Since the sets $\widehat{S}_p$ and $\widehat{H}_p$ contain a finite number of elements, they are clearly closed nonempty sets. To characterize the projector onto $\widehat{S}_p$ observe that
\begin{align}
e_{i_0}+e_{j_0} \in P_{\widehat{S}_p}(x) \quad  \iff \quad & \|x-e_{i_0}-e_{j_0}\|\leq \|x-e_i-e_j\|, \quad \text{for all } i,j\in\{1,2,\ldots,p\}\text{~with~}i\neq j,\nonumber\\[1ex]
\iff \quad &  \|x-e_{i_0}-e_{j_0}\|^2\leq \|x-e_i-e_j\|^2, \quad \text{for all } i,j\in\{1,2,\ldots,p\}\text{~with~} i\neq j,\nonumber\\[1ex]
\iff \quad & \|x\|^2+2-2x^Te_{i_0}-2x^Te_{j_0}\leq \|x\|^2+2-2x^Te_{i}-2x^Te_{j},\nonumber\\
& \text{for all } i,j\in\{1,2,\ldots,n\}, i\neq j,\nonumber\\[1ex]
\iff \quad & x_i+x_j\leq x_{i_0}+x_{j_0}, \quad \text{for all } i,j\in\{1,2,\ldots,n\}, i\neq j,\nonumber\\[1ex]
\iff \quad & x_i\leq x_{i_0} \text{ and } x_i\leq x_{j_0}, \quad \text{for all } i\in\{1,2,\ldots,n\}.\label{eq:PBS}
\end{align}
By using a case distinction and following analogous reasoning, we can easily obtain that
\begin{subequations}\label{eq:PBH}
\begin{align}
0_p \in P_{\widehat{H}_p}(x) \quad & \iff \quad x_i\leq \frac{1}{2}, \text{ for all } i\in\{1,2,\ldots,p\},\label{eq:PBn2_0}\\[1ex]
e_{i_0} \in P_{\widehat{H}_p}(x) \quad & \iff \quad x_j\leq\frac{1}{2}\leq x_{i_0}, \text{ for all } j\in\{1,2,\ldots,p\}\setminus\{i_0\};\label{eq:PBn2_1}\\[1ex]
e_{i_0}+e_{j_0} \in P_{\widehat{H}_p}(x) \quad & \iff \quad \begin{array}{c}
\text{with~}\frac{1}{2}\leq x_{i_0}, \quad \frac{1}{2}\leq x_{j_0},\\[1ex]
x_i\leq x_{i_0},\,\, x_i\leq x_{j_0}, \text{ for all } i\in\{1,2,\ldots,p\}\setminus\{i_0,j_0\}.
\end{array}\label{eq:PBn2_2}
\end{align}
\end{subequations}
Finally, observe that \eqref{eq:piBS} and \eqref{eq:piBH} are directly derived from \eqref{eq:PBS} and \eqref{eq:PBH}, respectively.
\end{proof}

\begin{remark}[On the multivaluedness of the projectors]
In contrast to~\Cref{prop:projH}, where the projectors onto $S_p$ and $H_p$ were given, we only use the qualifier ``a projection'' in~\Cref{prop:projB}. Since the sets $\widehat{S}_p$ and $\widehat{H}_p$ are nonconvex, the projections may not be unique. Multivaluedness occurs when either there is a tie in the largest components of the vector to be projected, or when some of them are exactly $0.5$. Nonetheless, none of these situations will happen very often in practice, due to the finite precision of the computations, specially if the algorithm is run from a randomly generated starting point.
\end{remark}

Observe that the constraint sets $C_1$ and $C_2$ (resp. $\widehat{C}_1$ and $\widehat{C}_2$) determine that each row and each column of $X$, respectively, lays in the set $S_n$ (resp. $\widehat{S}_n$). On the other hand,  $C_3$ and $C_4$ (resp. $\widehat{C}_3$ and $\widehat{C}_4$) determine that each forward and each backward diagonal of $X$, respectively, belongs to $H_p$ (resp. $\widehat{H}_p$) for some $p\in\{1,2,\ldots,n\}$. Thus, we can compute projections onto all these sets by using~\Cref{prop:projH} (resp.~\Cref{prop:projB}).

\begin{remark}[Fixed points and solutions]\label{rem:fixedP_solutions}
As demonstrated in~\Cref{fig:DR_AP_nonconvex}, the method of cyclic projections often fails to be useful when applied to combinatorial problems as there exist fixed points of the operator which do not provide a solution. Consider, for instance, one iteration of this method, for the $(2,3)$-queens problem generated with Formulation~\hyperref[For3]{3} as

{\arraycolsep=3pt
\renewcommand{\arraystretch}{1}
$$
X_0:=\left( \begin{array}{ccc}
0 & 1 & 0\\
1 & 1 & 1\\
1 & 0 & 1\\
\end{array} \right)\xrightarrow[]{P_{\widehat{C}_1}} \left( \begin{array}{ccc}
0 & 1 & 1\\
0 & 1 & 1\\
1 & 0 & 1\\
\end{array} \right)\xrightarrow[]{P_{\widehat{C}_2}} \left( \begin{array}{ccc}
0 & 1 & 0\\
1 & 1 & 1\\
1 & 0 & 1\\
\end{array} \right)\xrightarrow[]{P_{{C}_3}} \left( \begin{array}{ccc}
0 & 1 & 0\\
1 & 1 & 1\\
1 & 0 & 1\\
\end{array} \right)\xrightarrow[]{P_{{C}_4}} \left( \begin{array}{ccc}
0 & 1 & 0\\
1 & 1 & 1\\
1 & 0 & 1\\
\end{array} \right)=X_0.
$$}%
Here we remark that projections onto $\widehat{C}_1$ and $\widehat{C}_2$ are computed using~\Cref{prop:projB}, so in case of tie between the two largest components of a row or a column, we choose those ones of largest indexes. The point $X_0$ is therefore a fixed point in the sense that
$$X_0\in P_{{C}_4}P_{{C}_3}P_{\widehat{C}_2}P_{\widehat{C}_1}(X_0).$$
However, $X_0$ is not a solution to the problem since the second row contains three queens.

The same problem occurs when we use a cyclic variant of the DR algorithm. For instance, we have
$$ Y_0\in T_{[\widehat{C}_1,\widehat{C}_2,C_3,C_4]}(Y_0) \quad \text{and} \quad Y_0\in T_{\widehat{C}_1,[\widehat{C}_2,C_3,C_4]}(Y_0),\quad \text{for } Y_0:=\left( \begin{array}{ccc}
0 & 1 & 1\\
1 & 1 & 0\\
0 & 1 & 1\\
\end{array} \right),$$
where $T_{[\widehat{C}_1,\widehat{C}_2,C_3,C_4]}$ and $T_{\widehat{C}_1,[\widehat{C}_2,C_3,C_4]}$ are the cyclic and the cyclically anchored DR operators defined in~\eqref{eq:DRoperator_cyclic} and~\eqref{eq:DRoperator_anchored}, respectively. When we use a cyclic DR method, the point that solves the problem is not the fixed point itself, but its projections onto any of the constraint sets. Observe that in this case
$$P_{\widehat{C}_1}(Y_0)=P_{{C}_3}(Y_0)=P_{C_4}(Y_0)=\{Y_0\} \quad \text{and} \quad P_{\widehat{C}_2}(Y_0)\cap \widehat{C}_1=\emptyset,$$
so a solution cannot be obtained from $Y_0$. The existence of this undesired type of fixed points cannot happen to the DR algorithm in the product space, thanks to \Cref{exercise:DRfixedpoints-set}.
\end{remark}

\paragraph{Numerical experiment.} We performed a numerical experiment to compare the behavior of the DR algorithm on the four formulations presented for the $(2,n)$-queens problem. Note that all the feasibility problems considered above are described by more than two sets, so we have to turn to the product space. We shall not present the results of running the method of cyclic projections and the other versions of the DR algorithm presented in \Cref{subsec:ProductSpace}, because all of them are highly unsuccessful in finding a solution, mainly due to~\Cref{rem:fixedP_solutions}. Therefore, for each of the four formulations defined by $r$ sets, we use the following implementation:

\begin{itemize}
\item \emph{Initialization}:~Generate a random vector $y\in\Bnn$ and set $x_{1,0}:=y,x_{2,0}:=y,\ldots,x_{r,0}:=y$.
\item \emph{Iteration}:~For $k=0, 1, 2,\ldots,$ compute $p_k$ and update $x_{1,k+1}, x_{2,k+1},\ldots, x_{r,k+1}$ according to~\eqref{eq:DR_productspace}.
\item \emph{Termination}:~The algorithm is terminated if either a solution is found or a maximum time of 300 seconds has been reached. In order to test if the current iterate provides a solution, we pointwise round the current shadow point, $p_k$, to the nearest integer and check if this integer point satisfies all constraints.
\end{itemize}
All implementations were coded in Python~2.7 and run on an Intel Core i7-4770 CPU \@3.40 GHz with 16 GB RAM running Windows 10 (64-bit).

Our experiments were performed with chessboards of size $n$ for $n\in \{10, 20, \ldots, 100\}$. For each size and each formulation, the DR algorithm was run from $20$ random starting points. The results of the experiment are summarized in~\Cref{fig:8Qresults}. We observe that Formulations \hyperref[For1]{1} and \hyperref[For2]{2} both performed poorly, only being able to solve small sized boards. Formulation~\hyperref[For4]{4} performed well, however Formulation~\hyperref[For3]{3} was the best with at most one failed instance for all sizes. Again, we wish to emphasize that it is quite remarkable that the DR method works here at all, let alone so well, given the absence of any sound theoretical justification for this behavior.

\begin{figure}[hbt]
	\centering
	\includegraphics[width=0.8\textwidth]{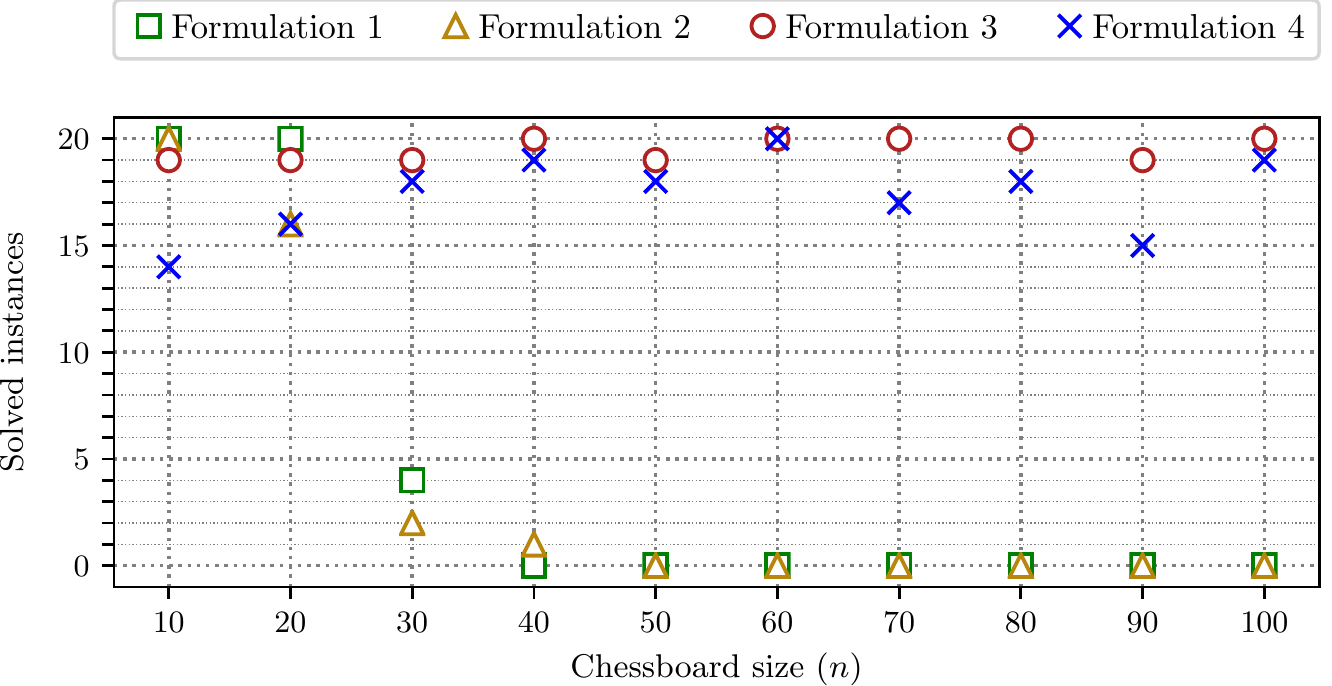}\\[3ex]
	\includegraphics[width=0.8\textwidth]{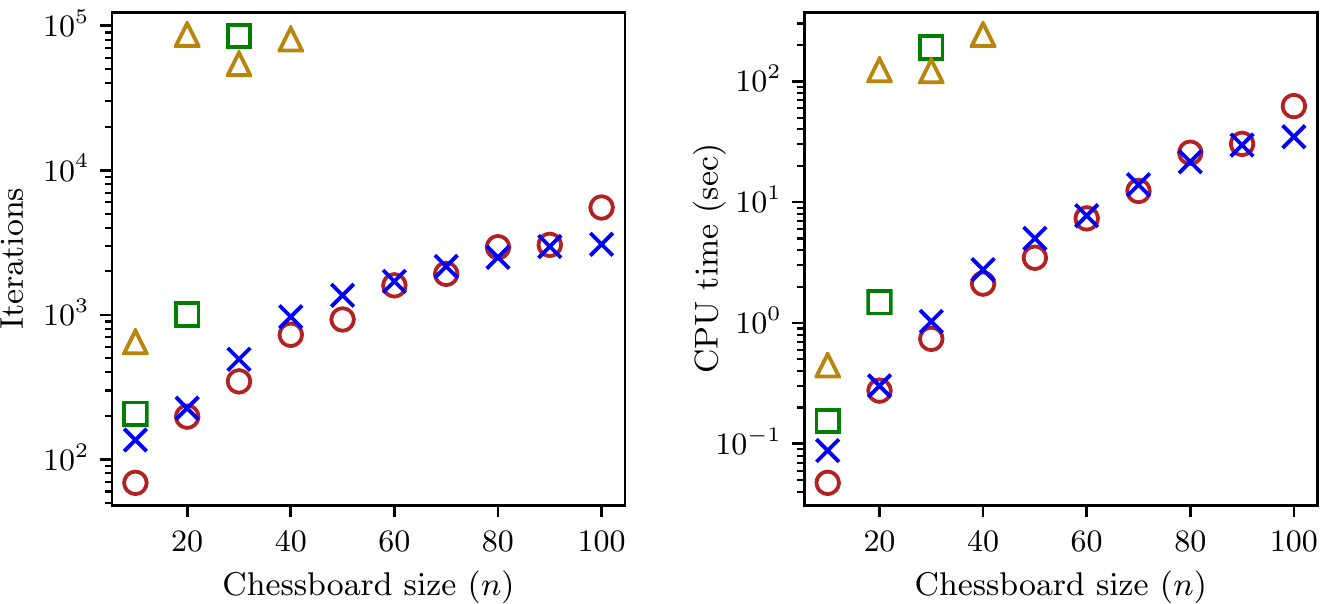}
	\caption{Results from the $(2,n)$-queens experiment. The top figure shows the number of instances (out of 20) solved by DR in less than $300$ seconds. The bottom figures plots the average number of iterations (left) and average CPU time (right) over the solved instances for each formulation and each board size.}\label{fig:8Qresults}
\end{figure}

\section{The infinite-dimensional case}\label{sec:infinite}
In order to avoid the topological technicalities associated with working in an infinite dimensional space, we have so far restricted our setting to a Euclidean space (i.e., a finite dimensional Hilbert space). Nevertheless, most of the presented theory on fixed-point iterations and projection algorithms remain valid in a potentially infinite-dimensional Hilbert space, $\Hi$, with \emph{strong convergence} of sequences replaced by \emph{weak convergence}.

\begin{definition}[Weak/strong convergence]
Consider a sequence $(x_k)_{k=0}^\infty\subset\Hi$ and a point $x^\star\in\Hi$. Then $(x_k)_{k=0}^\infty\subset\Hi$ is said to:
\begin{enumerate}
\item \emph{converge strongly} to $x^\star$, denoted $x_k\to x^\star$, if $\|x_k-x^*\|\to 0$ as $k\to\infty$.
\item \emph{converge weakly} to $x^\star$, denoted $x_k\wto x^\star$, if
\begin{equation}\label{eq:weak convergence}
\langle x_k,z\rangle \to \langle x^\star,z\rangle\text{~~for all~} z\in\Hi.
\end{equation}
\end{enumerate}
\end{definition}
In general, strong convergence implies weak convergence, whereas the opposite is false. In fact, the \emph{Kadec--Klee property} (\Cref{exercise:kadecklee}) gives a precise contain between the two notions. Moreover, when $\Hi$ is finite dimensional, both notions coincide, as can be seen by applying \eqref{eq:weak convergence} with $z$ equal to each element of an orthonormal basis for the space.

\begin{exercise}[Kadec--Klee property]\label{exercise:kadecklee}
 Consider a sequence $(x_k)\subseteq\Hi$ and a point $x^\star\in\Hi$. Show that $x_k\to x^\star$ if and only if $x_k\wto x^\star$ and $\|x_k\|\to\|x^\star\|$.
\end{exercise}	

The main technical difficulty arising from weak convergence is that projectors onto closed and convex sets are generally not weakly (sequentially) continuous, as the following example shows. Note there is no ambiguity here in using the term ``closed'' without qualification for convex sets, since a convex set is weakly (sequentially) closed if and only if it is strongly (sequentially) closed \cite[Theorem~3.34]{BC17}.

\begin{example}[Projectors are not weakly continuous {\cite[Example~4.20]{BC17}}]\label{example:weak cont}
Consider the infinite dimensional Hilbert space  of square summable real sequences
 $$\mathcal{H}:=\ell_2(\mathbb{N})=\left\{x=(x_i)_{i=1}^\infty\subset \R:\|x\|^2=\sum_{i=1}^\infty x_i^2<+\infty\right\},$$
endowed with the inner product $\langle x,y\rangle=\sum_{i=1}^\infty x_iy_i$.  Denote its closed unit ball by $\B:=\{x\in\mathcal{H}:\|x\|\leq 1\}$. Consider the bounded sequence $(e^k)_{k=2}^\infty\subset\mathcal{H}$ where, for all $k\geq 1$, the sequence $e^k=(e^k_i)_{i=1}^\infty\subset \R$ is given by
  $$e^k_i := \begin{cases} 1 & \text{if~}i\in\{1,k\}, \\
                          0 & \text{otherwise.} \\ \end{cases}$$
Then $e^k\wto e^1$ as $k\to\infty$ and $P_\B(e^1)=e^1$. However, $ P_\B(e^k) = \frac{1}{\sqrt{2}}e^k\wto \frac{1}{\sqrt{2}}e^1 \neq P_\B(e^1)$.\qede
\end{example}
\begin{exercise}
Verify the details of \Cref{example:weak cont}. That is, show that $e^k\wto e^1$  but $P_\B(e^k)\not\wto P_\B(e^1)$.
\end{exercise}

As pointed out in \Cref{rem:DR_shadow}, when applying the DR algorithm, the sequence of interest is the shadow sequence $(P_A(x_k))_{k=0}^{\infty}$. For the reasons given in
\Cref{example:weak cont}, weak convergence of this sequence cannot be directly derived from weak converge of $(x_k)_{k=0}^{\infty}$, since the projector $P_A$
may not be weakly sequentially continuous. Nevertheless, \cite{Svaiter} proved that the shadow sequence is indeed weakly convergent to $P_A(x^\star)$, where $x^\star$ denotes the weak limit of $(x_k)_{k=0}^{\infty}$. A simpler proof of this result was later provided by~\cite{Ba13} by means of a \emph{demiclosedness principle}.

\subsection{A non-negative moment problem}\label{example:moments}
Consider the infinite dimensional Hilbert space of square-integrable real-valued functions on the interval $[a,b]$ given by
$$\Hi:=L^2([a,b])=\left\{x:[a,b]\to\R: \int_{a}^bx(t)^2dt<\infty\right\},$$
endowed with the inner product $\langle x,y\rangle=\int_{a}^bx(t)y(t)dt$.
Suppose that we are interested in finding a probability density function $x\in\Hi$ with mean $\mu$ and variance $\sigma^2$. Hence, we aim to find a non-negative function $x\in\Hi$ that satisfies
the underdetermined linear system
$$\langle x,1\rangle=1,\quad \langle x,t\rangle=\mu\quad \text{and}\quad \langle x,t^2\rangle=\sigma^2+\mu^2.$$
Defining $C_i:=\left\{ x\in \Hi: \langle x,t^{i-1}\rangle=c_i\right\}$ for $i=1,2,3$ with $c_1 :=1$, $c_2:=\mu$ and $c_3:=\sigma^2+\mu^2$, and $C_4:=\{x\in\Hi: x\geq 0\}$, we observe that the problem is an infinite dimensional modification of \cref{example:linear_systems} (where we have added the nonnegativity constraint). The projectors onto $C_i$ with $i=1,2,3$ are given by
$$P_{C_i}(x)(t)=x(t)+\left(c_i-\langle t^{i-1},x\rangle\right)\frac{t^{i-1}}{\|t^{i-1}\|^2}=x(t)+\left(c_i-\int_{a}^bt^{i-1}x(t)dt\right)\frac{(2i-1)t^{i-1}}{b^{2i-1}-a^{2i-1}},$$
whereas
$$P_{C_4}(x)(t)=\max \{0, x(t)\}.$$

In \Cref{fig:moments_a,fig:moments_b,fig:moments_c,fig:moments_d}, we show the result of applying four of the projection methods presented in \cref{subsec:ProductSpace} to the problem of finding a probability density function defined on $[0,1]$, with mean $\mu=\frac{1}{2}$ and variance $\sigma^2=\frac{1}{20}$, using $x_0=1$ as starting point for all the algorithms. These computations were performed symbolically, without discretizing the problem.

\begin{figure}[ht!]
\centering
\subfigure[Method of cyclic projections]{\label{fig:moments_a}
\includegraphics[width=0.43\textwidth]{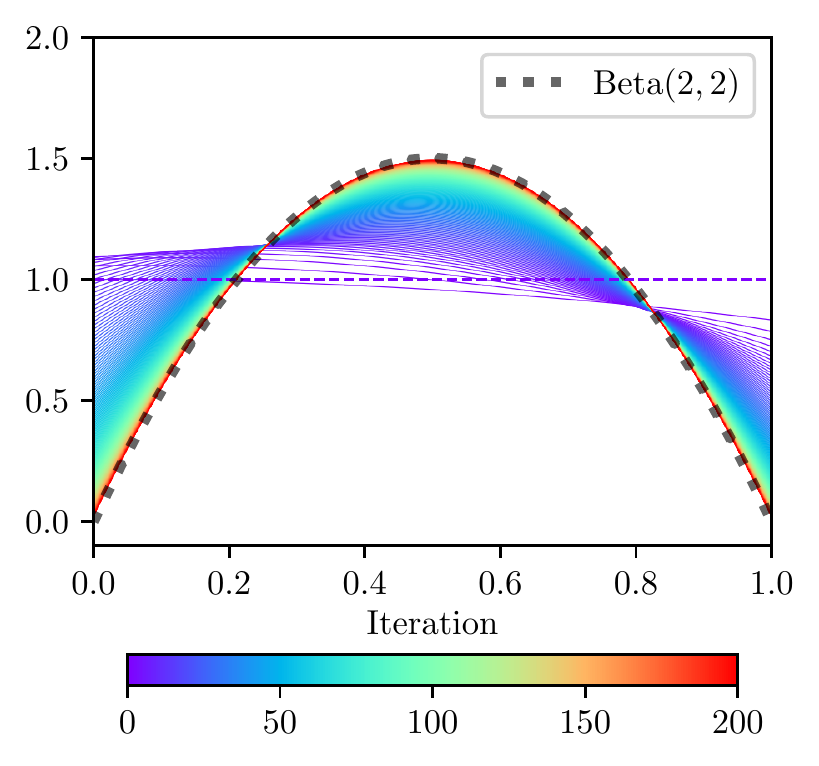}}
\hspace{0.001\textwidth}
\subfigure[Cyclic Douglas--Rachford] {\label{fig:moments_b}
\includegraphics[width=0.43\textwidth]{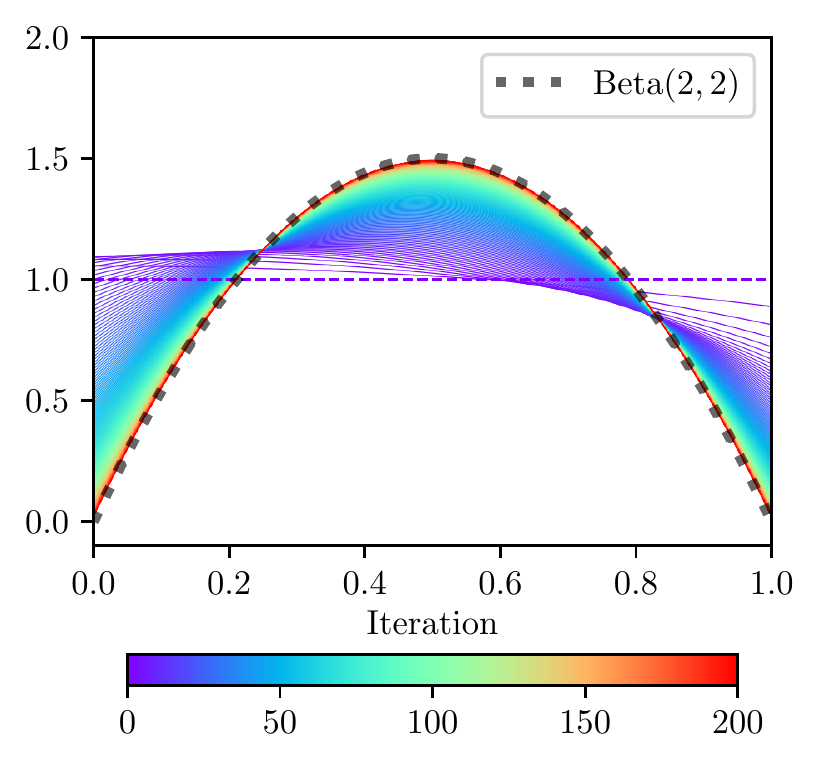}}
\subfigure[Cyclically anchored Douglas--Rachford] {\label{fig:moments_c}
\includegraphics[width=0.43\textwidth]{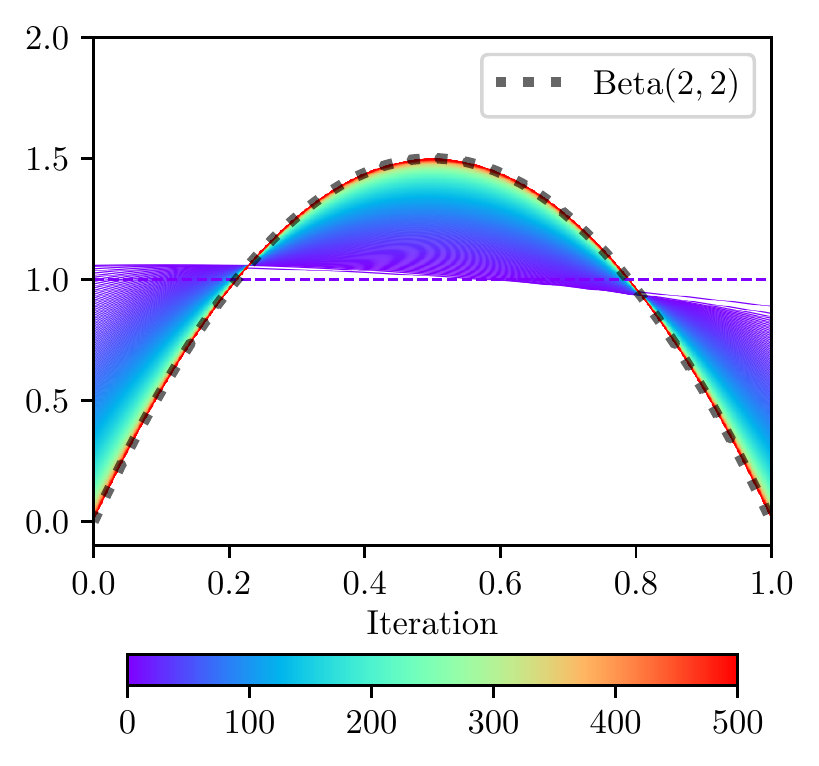}}
\hspace{0.001\textwidth}
\subfigure[Douglas--Rachford in the product space] {\label{fig:moments_d}
\includegraphics[width=0.43\textwidth]{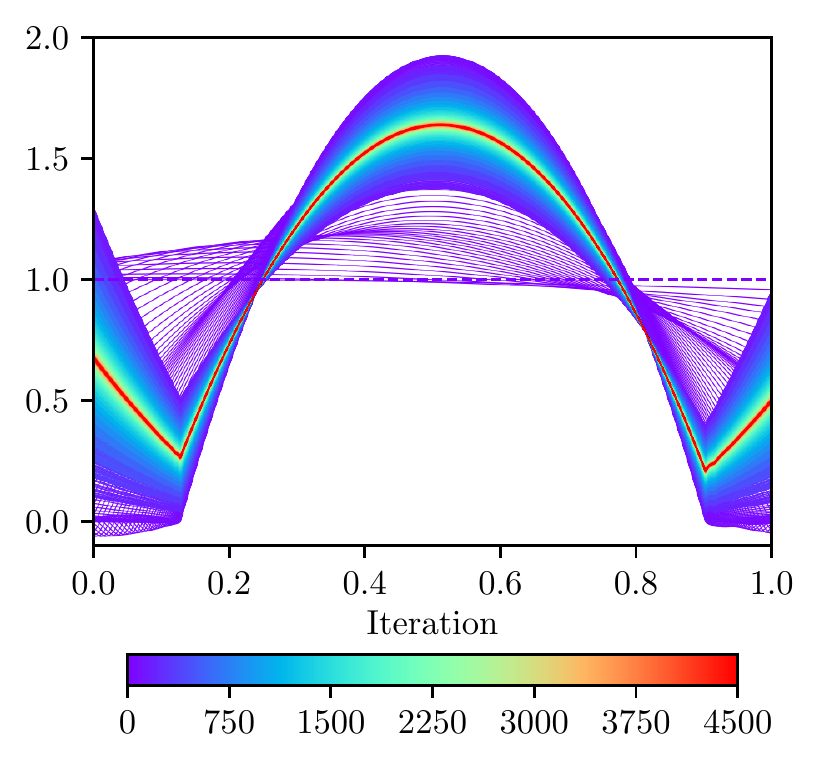}}
\subfigure[Method of cyclic projections] {\label{fig:moments_e}
\includegraphics[width=0.43\textwidth]{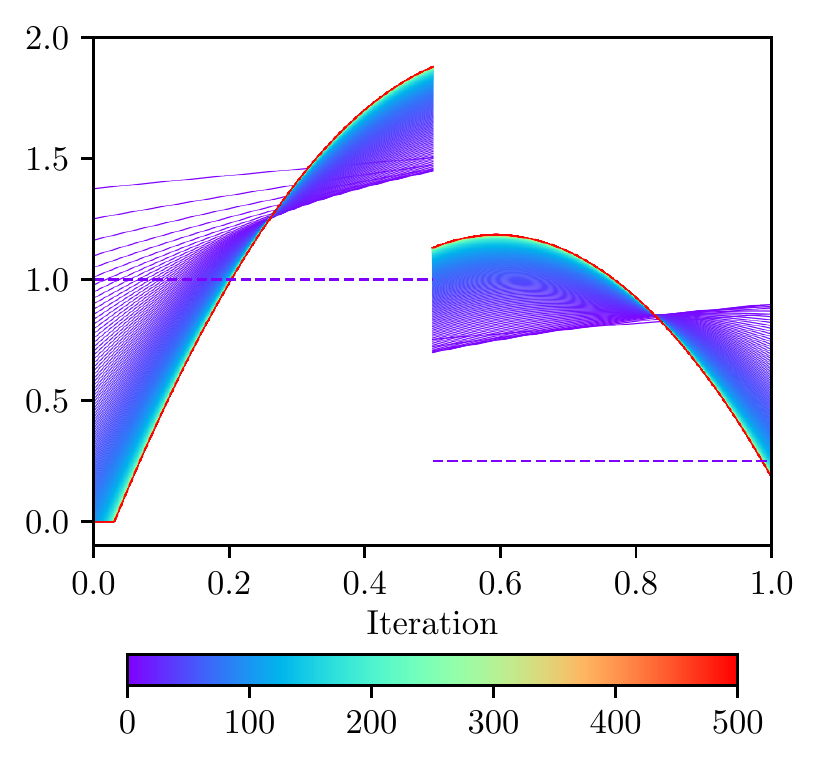}}
\hspace{0.001\textwidth}
\subfigure[Averaged Alternating Modified Reflections] {\label{fig:moments_f}
\includegraphics[width=0.43\textwidth]{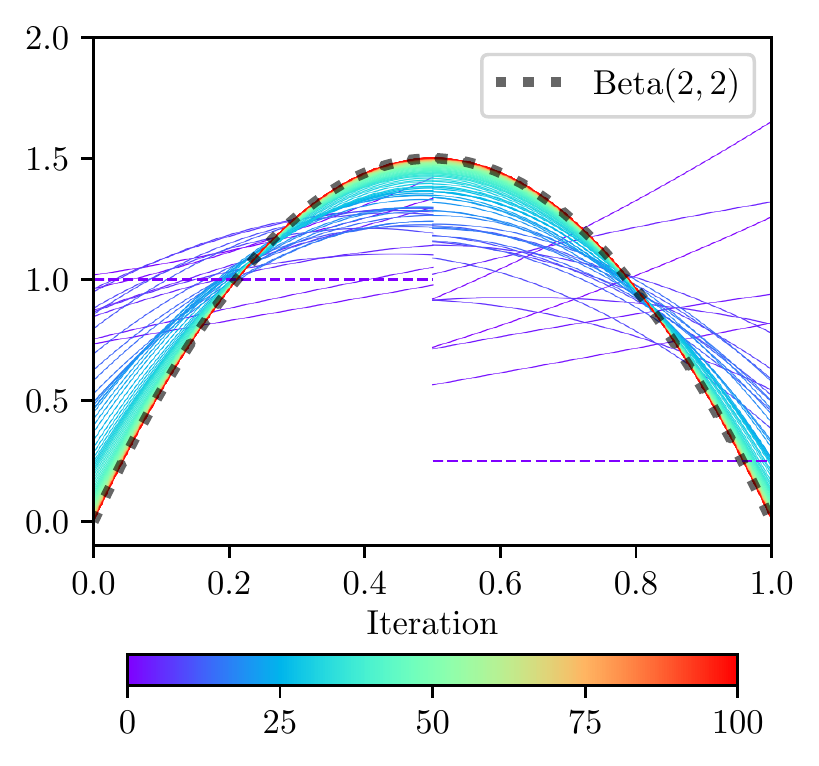}}
\caption{Sequences generated by various projection algorithms for the non-negative moment problem with $a=0$, $b=1$, $\mu=1/2$ and $\sigma^2=1/20$. The starting point (function) is represented with a dashed line.}\label{fig:Moments}
\end{figure}

On one hand, we clearly observe that the method of cyclic projections, the cyclic DR algorithm, and the cyclically anchored DR algorithm all converge to a Beta distribution with parameters $(2,2)$. Of these three methods, the cyclically anchored DR algorithm was the slowest. On the other hand, we realize that the behavior of the sequence generated by the DR algorithm in the product space is completely different to that of the other methods: the algorithm converges to a different (nonsmooth) solution to the problem. As it was also (by far) the slowest algorithm among the four, the DR algorithm in the product space seems to be the worst choice for this problem. Note none of these four algorithms guarantee strong convergence in this infinite dimensional setting.

In \cref{fig:moments_e}, we show the result of running the method of cyclic projections from a different starting point, given by
$$\widetilde{x}_0(t):=\begin{cases} 1 & \text{if~}0\leq t\leq \frac{1}{2}, \\
                          \frac{1}{4} & \text{if~}\frac{1}{2}< t\leq 1. \\ \end{cases}$$
This time, the method converges to a different (discontinuous) solution. In \cref{fig:moments_f}, we present the result of running the averaged alternating modified reflections method in the product space, with $q=0$ and parameters $\alpha=\beta=0.95$, using $\widetilde{x}_0$ as a starting point. The sequence converges in much less iterations than the ones generated by the other methods, even when the continuous starting point $x_0=1$ was used. Moreover, in contrast with the other methods, the shadow sequence is now guaranteed to be strongly convergent to a fixed point of the operator whenever the set of fixed points is nonempty \cite[according to][Theorem~5.1]{AAMR}. If a fixed point exists, its projection onto the diagonal $\bs{D}$ would solve the best approximation problem consisting in finding the probability density function on $[0,1]$ of minimum norm with $(\mu,\sigma^2)=\left(\frac{1}{2},\frac{1}{20}\right)$.

We observe in \cref{fig:moments_f} that the shadow sequence converges to the probability density function of a ${\rm Beta}(2,2)$. In \cref{fig:Moments_AAMR} we present the sequences in the product space resulting from running again the averaged alternating modified reflections method from the continuous starting point $x_0=1$.

\begin{figure}[ht!]
\centering
\subfigure[Sequences in the product space] {\includegraphics[height=0.45\textwidth]{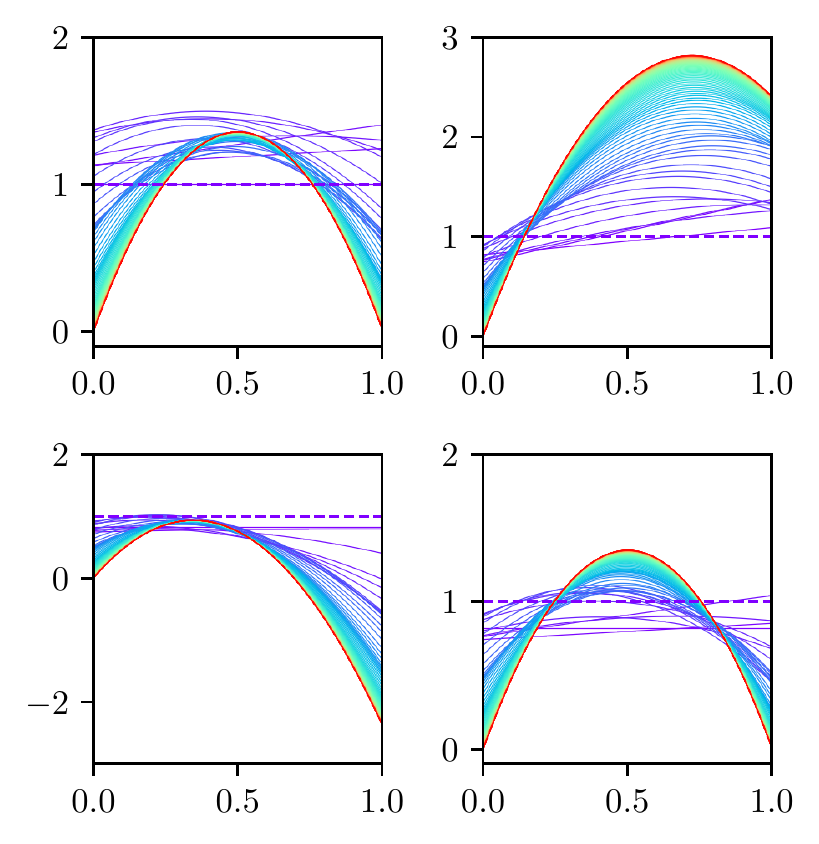}}
\subfigure[Shadow sequence in the diagonal space]{\includegraphics[height=0.45\textwidth]{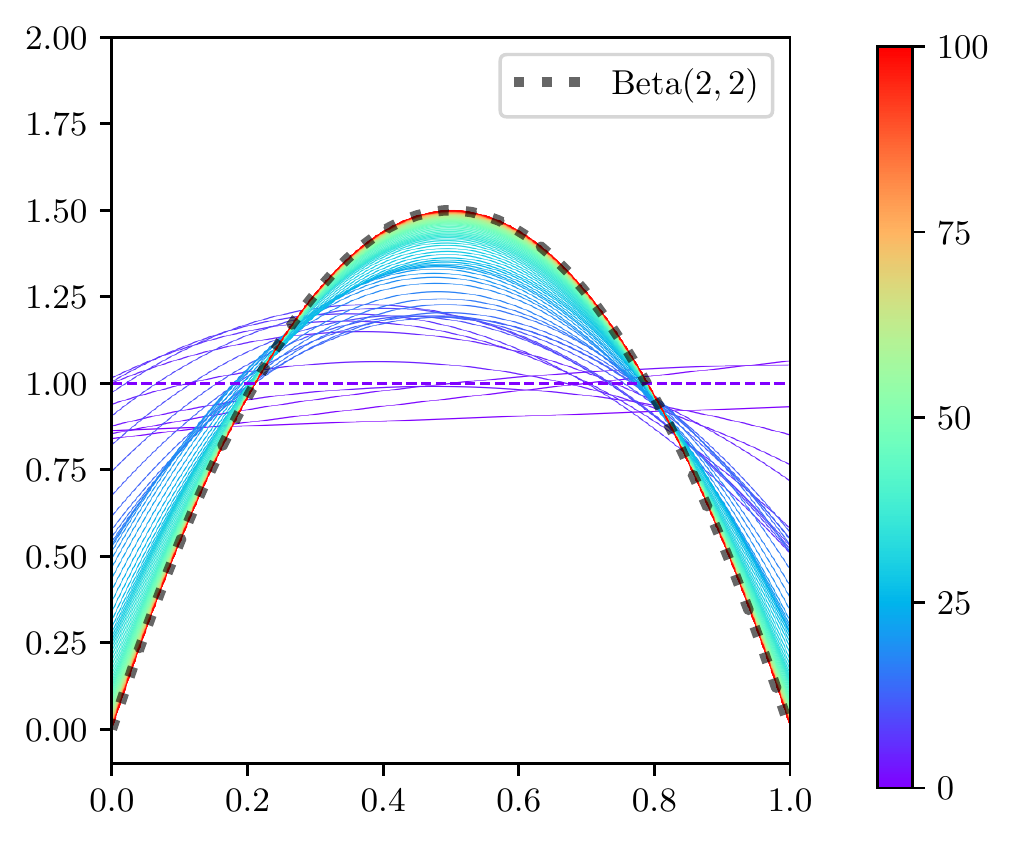}}
\caption{Sequences generated by the averaged alternating modified reflections method in the product space for the non-negative moment problem with $a=0$, $b=1$, $\mu=1/2$ and $\sigma^2=1/20$. The starting point (function) is represented with a dashed line.}\label{fig:Moments_AAMR}
\end{figure}

It is apparent that each of the four sequences generated is convergent to some function in $\Hi$. Looking at the expression of the sequences, it can be observed that these limiting functions are
\begin{align}\label{eq:barx}
 \overline{x}_1(t)&:=\frac{27}{5}t(1-t),\quad \overline{x}_2(t):=\frac{3}{5}t(13-9t),\quad \overline{x}_3(t):=\frac{3}{5}t(9-13t),\quad
 \overline{x}_4(t):=\frac{27}{5}t(1-t).
\end{align}
It can be checked that $\bs{\overline{x}}:=(\overline{x}_1,\overline{x}_2,\overline{x}_3,\overline{x}_4)\in
\bs{\Hi}=\Hi^4$ is a fixed point of the averaged alternating modified reflections operator $T_{\bs{D},\bs{C},0.95,0.95}$ given by~\eqref{eq:AAMR}. Therefore, since
$$P_{\bs{D}}(\bs{\overline{x}})(t)=\frac{\overline{x}_1+\overline{x}_2+\overline{x}_3
+\overline{x}_4}{4}=6t(1-t),$$
which coincides with the probability density function of a ${\rm Beta}(2,2)$, we can conclude that ${\rm Beta}(2,2)$ is the probability density function  of minimum norm in $L^2([0,1])$ with $(\mu,\sigma^2)=\left(\frac{1}{2},\frac{1}{20}\right)$.

\begin{exercise}
Verify that $\bs{\overline{x}}$ given by~\eqref{eq:barx} is a fixed point of $T_{\bs{D},\bs{C},0.95,0.95}$.
\end{exercise}

\section{Conclusions and open questions}\label{sec:conslusions}
In this self-contained tutorial, we introduced the feasibility problem paradigm and developed the convergence theory for projection algorithms within the framework of nonexpansive fixed point theory. The level of abstraction provided by this convergence framework allowed us to easily analyze several well-known projection algorithms in the literature. Particular attention was given to one member from the projection algorithm family, the Douglas--Rachford method, due to its known ability to solve certain combinatorial problems. By way of computational experiments on the $(2,n)$-queens problem, we demonstrated that choosing a good feasibility formulation is absolutely essential for this observed success. Although the product space formulation is a crucial ingredient of the DR method for feasibility problems with more than two constraint sets, we observed that it can actually make the method very slow, even in the convex setting with only four constraints as was considered in \Cref{sec:infinite}. We remark that the Douglas--Rachford projection algorithm, as was presented here, can be viewed as special case of an algorithm (which bears the same name) for finding a zero in the sum of two \emph{maximal monotone operators}. This was, in fact, the algorithm developed by~\cite{LM79}. For a modern treatment, see \cite[\S25.2]{BC17}.

To conclude this tutorial, we propose three open problems suitable as starting points for open-ended research projects, organized in order of decreasing difficulty.
\begin{itemize}
\item \emph{An explicit counterexample for weak convergence} --
As discussed in \Cref{sec:infinite}, in the infinite dimensional setting, the DR method is only known to converge weakly when applied to feasibility problems with convex sets. Despite this fact, to best of the authors' knowledge, there is no explicitly known example of two sets and an initial point for which the DR method converges weakly but not strongly. For the method of cyclic projections, a construction due to \cite{H2004} provides such a counterexample. The goal of this research project is to formulate an explicit counterexample for strong converge of the DR method.

\item \emph{Global convergence of DR for combinatorial problems} --
As stated in \Cref{sec:queens}, it is quite remarkable that the DR method applied to the $(m,n)$-queens problem works so well, especially given the absence of any applicable theory for the method in such settings. The goal of this research project is to find sufficient criteria for global convergence of the DR method involving more realistic nonconvex sets than the ones existing in the literature (e.g., involving the product space reformulation). Knowledge of such conditions would provide insights into general properties of useful feasibility formulations for combinatorial problems.

\item \emph{AAMR for combinatorial problems} -- As demonstrated in \Cref{fig:DR_AP_nonconvex} and \Cref{rem:fixedP_solutions}, most projection algorithms possess fixed points which are not necessary related to the feasibility problem at hand. The DR method however does not suffer this potential shortcoming (\Cref{exercise:DRfixedpoints-set}) and nor does the AAMR discussed in \Cref{sec:AAMR}. The latter suggests the AAMR method as a potentially useful heuristic. The goal of this project is to numerically compare the performance of the AAMR method to the DR method and GDR method on non-convex and combinatorial problems, such as the $(m,n)$-queens problem. A number of additional feasibility formulations for numerical testing can be found in \cite{ABTmatrix,ABTcomb}. Our initial numerical results for the $(2,n)$-queens problem suggest that the AAMR performs quite well (see \Cref{fig:AAMRcomb}).
\end{itemize}

\begin{figure}[bth]
	\centering
	\includegraphics[width=0.8\textwidth]{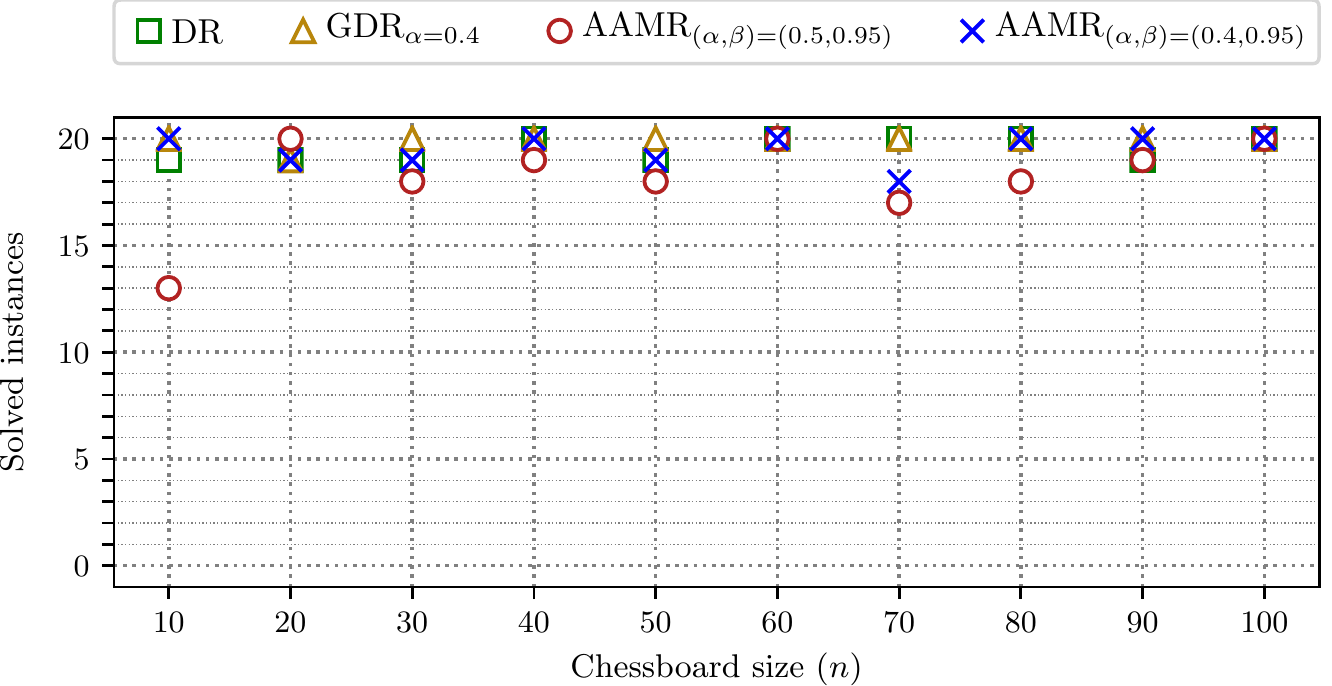}\\[3ex]
	\includegraphics[width=0.8\textwidth]{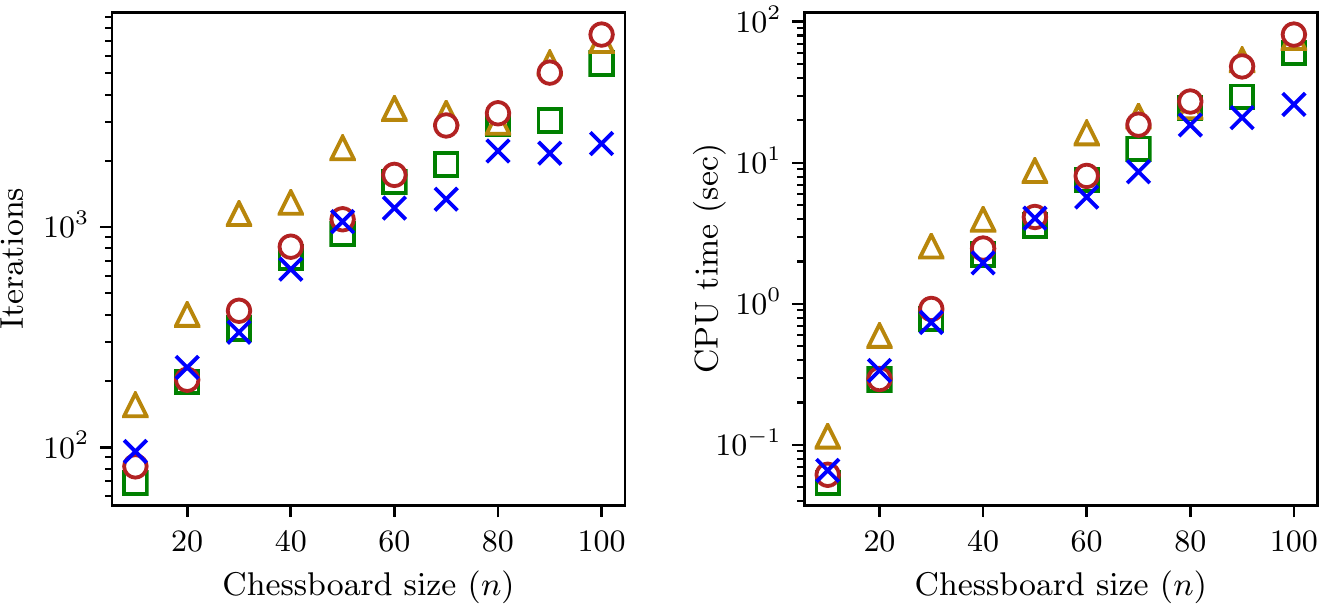}
	\caption{Results comparing DR, GDR and AAMR for the $(2, n)$-queens problem with Formulation~\hyperref[For3]{3}. The top figure shows the number of instances (out of 20) solved in less than $300$ seconds. The bottom figures plots the average number of iterations (left) and average CPU time (right) over the solved instances for each formulation and each board size.}\label{fig:AAMRcomb}
\end{figure}

{\small
\paragraph{Acknowledgments.}
FJAA and RC were partially supported by Ministerio de Econom\'ia, Industria y Competitividad (MINECO) and European Regional Development Fund (ERDF), grant MTM2014-59179-C2-1-P. FJAA was supported by the Ram\'on y Cajal program by MINECO and ERDF (RYC-2013-13327) and RC was supported by MINECO and European Social Fund (BES-2015-073360) under the program ``Ayudas para contratos predoctorales para la formaci\'on de doctores 2015''.}

\bibliographystyle{spbasic}      

\begin{thebibliography}{60}
\providecommand{\natexlab}[1]{#1}
\providecommand{\url}[1]{{#1}}
\providecommand{\urlprefix}{URL }
\expandafter\ifx\csname urlstyle\endcsname\relax
  \providecommand{\doi}[1]{DOI~\discretionary{}{}{}#1}\else
  \providecommand{\doi}{DOI~\discretionary{}{}{}\begingroup
  \urlstyle{rm}\Url}\fi
\providecommand{\eprint}[2][]{\url{#2}}
\small
\bibitem[{Alwadani et~al.(2018)Alwadani, Bauschke, Moursi, and
  Wang}]{AAMR_asymptotic}
Alwadani S, Bauschke HH, Moursi WM, Wang X (2018) On the asymptotic behaviour
  of the {A}rag\'on {A}rtacho--{C}ampoy algorithm. Oper Res Lett 46(6):585--587

\bibitem[{Arag\'on and Borwein(2013)}]{ABglobal}
Arag\'on~Artacho FJ, Borwein JM (2013) Global convergence of a non-convex
  {D}ouglas--{R}achford iteration. J Glob Optim 57(3):753--769

\bibitem[{Arag\'on and Campoy(2018{\natexlab{a}})}]{AAMR}
Arag\'on~Artacho FJ, Campoy R (2018{\natexlab{a}}) A new projection method for
  finding the closest point in the intersection of convex sets. Comput Optim
  Appl 69(1):99--132

\bibitem[{Arag\'on and Campoy(2018{\natexlab{b}})}]{Graphcol}
Arag\'on~Artacho FJ, Campoy R (2018{\natexlab{b}}) Solving graph coloring
  problems with the {D}ouglas--{R}achford algorithm. Set-Valued Var Anal
  26(2):277--304

\bibitem[{Arag\'on et~al.(2014{\natexlab{a}})Arag\'on, Borwein,
  and Tam}]{ABTmatrix}
Arag\'on~Artacho FJ, Borwein JM, Tam MK (2014{\natexlab{a}})
  {D}ouglas--{R}achford feasibility methods for matrix completion problems.
  ANZIAM J 55(4):299--326

\bibitem[{Arag\'on et~al.(2014{\natexlab{b}})Arag\'on, Borwein,
  and Tam}]{ABTcomb}
Arag\'on~Artacho FJ, Borwein JM, Tam MK (2014{\natexlab{b}}) Recent results on
  {D}ouglas--{R}achford methods for combinatorial optimization problem. J Optim
  Theory Appl 163(1):1--30

\bibitem[{Arag\'on et~al.(2016)Arag\'on, Borwein, and
  Tam}]{ABT16}
Arag\'on~Artacho FJ, Borwein JM, Tam MK (2016) Global behavior of the
  {D}ouglas--{R}achford method for a nonconvex feasibility problem. J Glob
  Optim 65(2):309--327

\bibitem[{Arag\'on et~al.(2018{\natexlab{a}})Arag\'on, Campoy,
  and Elser}]{Graphcol2}
Arag\'on~Artacho FJ, Campoy R, Elser V (2018{\natexlab{a}}) An enhanced
  formulation for successfully solving graph coloring problems with the
  {D}ouglas--{R}achford algorithm. arXiv e-prints
  \eprint{https://arxiv.org/abs/1808.01022}

\bibitem[{Arag\'on et~al.(2018{\natexlab{b}})Arag\'on, Campoy,
  Kotsireas, and Tam}]{Designs}
Arag\'on~Artacho FJ, Campoy R, Kotsireas IS, Tam MK (2018{\natexlab{b}}) A
  feasibility approach for constructing combinatorial designs of circulant
  type. J Comb Optim 35(4):1061--1085

\bibitem[{Arag\'on et~al.(2018{\natexlab{c}})Arag\'on, Censor,
  and Gibali}]{CG18}
Arag\'on~Artacho FJ, Censor Y, Gibali A (2018{\natexlab{c}}) The cyclic
  {D}ouglas--{R}achford algorithm with $r$-sets-{D}ouglas--{R}achford
  operators. Optim Methods Softw \doi{10.1080/10556788.2018.1504049}

\bibitem[{Baillon et~al.(1978)Baillon, Bruck, and Reich}]{BBR78}
Baillon JB, Bruck RE, Reich S (1978) On the asymptotic behavior of nonexpansive
  mappings and semigroups in {B}anach spaces. Houston J Math 4(1):1--9

\bibitem[{Banach(1922)}]{banach1922operations}
Banach S (1922) Sur les op{\'e}rations dans les ensembles abstraits et leur
  application aux {\'e}quations int{\'e}grales. Fund math 3(1):133--181

\bibitem[{Bauschke(2013)}]{Ba13}
Bauschke HH (2013) New demiclosedness principles for (firmly) nonexpansive
  operators. In: Computational and Analytical Mathematics, Springer, pp 19--28

\bibitem[{Bauschke and Combettes(2017)}]{BC17}
Bauschke HH, Combettes PL (2017) Convex Analysis and Monotone Operator Theory
  in {H}ilbert Spaces, 2nd edn. Springer

\bibitem[{Bauschke and Dao(2017)}]{BD17}
Bauschke HH, Dao MN (2017) On the finite convergence of the
  {D}ouglas--{R}achford algorithm for solving (not necessarily convex)
  feasibility problems in {E}uclidean spaces. SIAM J Optim 27(1):507--537

\bibitem[{Bauschke and Moursi(2017)}]{BM17}
Bauschke HH, Moursi WM (2017) On the {D}ouglas--{R}achford algorithm. Math
  Program, Ser A 164(1--2):263--284

\bibitem[{Bauschke and Noll(2014)}]{BNlocal}
Bauschke HH, Noll D (2014) On the local convergence of the
  {D}ouglas--{R}achford algorithm. Arch Math 102(6):589--600

\bibitem[{Bauschke et~al.(2002)Bauschke, Combettes, and Luke}]{BCL02}
Bauschke HH, Combettes PL, Luke DR (2002) Phase retrieval, error reduction
  algorithm, and {F}ienup variants: a view from convex optimization. J Opt Soc
  Am A 19(7):1334--1345

\bibitem[{Bauschke et~al.(2004)Bauschke, Combettes, and Luke}]{BCL04}
Bauschke HH, Combettes PL, Luke DR (2004) Finding best approximation pairs
  relative to two closed convex sets in {H}ilbert spaces. J Approx Theory
  127(2):178--192

\bibitem[{Bauschke et~al.(2014)Bauschke, Bello~Cruz, Nghia, Phan, and
  Wang}]{BCNPW14}
Bauschke HH, Bello~Cruz JY, Nghia TT, Phan HM, Wang X (2014) The rate of linear
  convergence of the {D}ouglas--{R}achford algorithm for subspaces is the
  cosine of the {F}riedrichs angle. J Approx Theory 185:63--79

\bibitem[{Bauschke et~al.(2015)Bauschke, Noll, and Phan}]{anchored}
Bauschke HH, Noll D, Phan HM (2015) Linear and strong convergence of algorithms
  involving averaged nonexpansive operators. J Math Anal Appl 421(1):1--20

\bibitem[{Bauschke et~al.(2017)Bauschke, Lukens, and Moursi}]{BLM17}
Bauschke HH, Lukens B, Moursi WM (2017) Affine nonexpansive operators,
  {A}ttouch--{T}h{\'e}ra duality and the {D}ouglas--{R}achford algorithm.
  Set-Valued Var Anal 25(3):481--505

\bibitem[{Bauschke et~al.(2019)Bauschke, Dao, and Lindstrom}]{BDLdoubleton}
Bauschke HH, Dao MN, Lindstrom SB (2019) The {D}ouglas--{R}achford algorithm
  for a hyperplane and a doubleton. J Glob Optim pp 1--15,
  \doi{10.1007/s10898-019-00744-7}

\bibitem[{Behling et~al.(2018)Behling, Bello~Cruz, and Santos}]{BBS18}
Behling R, Bello~Cruz JY, Santos L (2018) Circumcentering the
  {D}ouglas--{R}achford method. Numer Algor 78(3):759--776

\bibitem[{Benoist(2015)}]{benoist}
Benoist J (2015) The {D}ouglas--{R}achford algorithm for the case of the sphere
  and the line. J Global Optim 63(2):363--380

\bibitem[{Borwein and Sims(2011)}]{BS11}
Borwein JM, Sims B (2011) The {D}ouglas--{R}achford algorithm in the absence of
  convexity. In: Bauschke H, Burachik R, Combettes P, Elser V, Luke D,
  Wolkowicz H (eds) Fixed-Point Algorithms for Inverse Problems in Science and
  Engineering, Springer Optimization and Its Applications, vol~49, Springer,
  New York, pp 93--109

\bibitem[{Borwein and Tam(2014)}]{BT14}
Borwein JM, Tam MK (2014) A cyclic {D}ouglas--{R}achford iteration scheme. J
  Optim Theory Appl 160(1):1--29

\bibitem[{Borwein and Tam(2015)}]{BT15}
Borwein JM, Tam MK (2015) The cyclic {D}ouglas--{R}achford method for
  inconsistent feasibility problems. J Nonlinear Convex Anal 16(4):573--584

\bibitem[{Borwein and Tam(2017)}]{BT17}
Borwein JM, Tam MK (2017) Reflection methods for inverse problems with
  applications to protein conformation determination. In: Aussel D, Lalitha C
  (eds) Generalized Nash Equilibrium Problems, Bilevel Programming and MPEC.
  Forum for Interdisciplinary Mathematics, Springer Singapore, Singapore, pp
  83--100

\bibitem[{Borwein et~al.(2018)Borwein, Lindstrom, Sims, Schneider, and
  Skerritt}]{BLSSS18}
Borwein JM, Lindstrom SB, Sims B, Schneider A, Skerritt MP (2018) Dynamics of
  the {D}ouglas--{R}achford method for ellipses and $p$-spheres. Set-Valued Var
  Anal 26(2):385--403

\bibitem[{Bregman(1965)}]{Breg65}
Bregman LM (1965) The method of successive projection for finding a common
  point of convex sets. Soviet Math Dokl 162(3):688--692

\bibitem[{Cegielski(2012)}]{C12}
Cegielski A (2012) Iterative Methods for Fixed Point Problems in {H}ilbert
  Spaces, Lecture Notes in Mathematics, vol 2057. Springer

\bibitem[{Censor(1984)}]{C84}
Censor Y (1984) Iterative methods for convex feasibility problems. Ann Discrete
  Math 20:83--91

\bibitem[{Censor and Cegielski(2015)}]{CC15}
Censor Y, Cegielski A (2015) Projection methods: an annotated bibliography of
  books and reviews. Optimization 64(11):2343--2358

\bibitem[{Dao and Tam(2019{\natexlab{a}})}]{DT18}
Dao MN, Tam MK (2019{\natexlab{a}}) A {L}yapunov-type approach to convergence
  of the {D}ouglas--{R}achford algorithm for a nonconvex setting. J Glob Optim
  73(1):83--112

\bibitem[{Dao and Tam(2019{\natexlab{b}})}]{dao2019union}
Dao MN, Tam MK (2019{\natexlab{b}}) Union averaged operators with applications
  to proximal algorithms for min-convex functions. J Optim Theory and Appl pp
  1--34

\bibitem[{Deutsch(2001)}]{D01}
Deutsch F (2001) Best Approximation in Inner Product Spaces, CMS Books in
  Mathematics/Ouvrages de Math\'ematiques de la SMC, vol~7. Springer-Verlag

\bibitem[{Douglas and Rachford(1956)}]{DR56}
Douglas J, Rachford HH (1956) On the numerical solution of heat conduction
  problems in two and three space variables. Trans Amer Math Soc 82:421--439

\bibitem[{Eckstein and Bertsekas(1992)}]{EB92}
Eckstein J, Bertsekas DP (1992) On the {D}ouglas--{R}achford splitting method
  and the proximal point algorithm for maximal monotone operators. Math Program
  55(1):293--318

\bibitem[{Elser(2003)}]{ElserPhase}
Elser V (2003) Phase retrieval by iterated projections. J Opt Soc Am A
  20(1):40--55

\bibitem[{Elser(2018)}]{ElserBit}
Elser V (2018) The complexity of bit retrieval. IEEE Transactions on
  Information Theory 64(1):412--428

\bibitem[{Elser et~al.(2007)Elser, Rankenburg, and Thibault}]{Elser}
Elser V, Rankenburg I, Thibault P (2007) Searching with iterated maps. Proc
  Natl Acad Sci 104(2):418--423

\bibitem[{Halperin(1962)}]{H62}
Halperin I (1962) The product of projection operators. Acta Sci Math 23:96--99

\bibitem[{Hesse and Luke(2013)}]{HLnonconvex}
Hesse R, Luke DR (2013) Nonconvex notions of regularity and convergence of
  fundamental algorithms for feasibility problems. SIAM J Optim
  23(4):2397--2419

\bibitem[{Hesse et~al.(2014)Hesse, Luke, and Neumann}]{HLNsparse}
Hesse R, Luke DR, Neumann P (2014) Alternating projections and
  {D}ouglas--{R}achford for sparse affine feasibility. IEEE Transactions on
  Signal Processing 62(18):4868--4881

\bibitem[{Hundal(2004)}]{H2004}
Hundal HS (2004) An alternating projection that does not converge in norm.
  Nonlin Anal: Theory, Methods \& Appl 57(1):35--61

\bibitem[{Kaczmarz(1937)}]{K37}
Kaczmarz S (1937) Angen\"aherte {A}ufl\"osung von {S}ystemen linearer
  {G}leichungen. Bull Int Acad Sci Pologne, A 35:355--357

\bibitem[{Lamichhane et~al.(2017)Lamichhane, Lindstrom, and Sims}]{LLS17}
Lamichhane BP, Lindstrom SB, Sims B (2017) Application of projection algorithms
  to differential equations: boundary value problems. arXiv e-prints
  \eprint{https://arxiv.org/abs/1705.11032}

\bibitem[{Lindstrom and Sims(2018)}]{LS18}
Lindstrom SB, Sims B (2018) Survey: Sixty years of {D}ouglas--{R}achford. arXiv
  e-prints \eprint{https://arxiv.org/abs/1809.07181}

\bibitem[{Lions and Mercier(1979)}]{LM79}
Lions PL, Mercier B (1979) Splitting algorithms for the sum of two nonlinear
  operators. SIAM J Numer Anal 16(6):964--979

\bibitem[{Luke(2008)}]{L08}
Luke DR (2008) Finding best approximation pairs relative to a convex and a
  prox-regular set in a {H}ilbert space. SIAM J Optim 19(2):714--739

\bibitem[{von Neumann(1950)}]{VN50}
von Neumann J (1950) Functional Operators II: The Geometry of Orthogonal
  Spaces. Princeton University Press

\bibitem[{Opial(1967)}]{opial1967weak}
Opial Z (1967) Weak convergence of the sequence of successive approximations
  for nonexpansive mappings. Bulletin of the American Mathematical Society
  73(4):591--597

\bibitem[{Pazy(1971)}]{P71}
Pazy A (1971) Asymptotic behavior of contractions in {H}ilbert space. Israel J
  Math 9:235--240

\bibitem[{Phan(2016)}]{Plinear}
Phan HM (2016) Linear convergence of the {D}ouglas--{R}achford method for two
  closed sets. Optim 65(2):369--385

\bibitem[{Pierra(1984)}]{Pierra}
Pierra G (1984) Decomposition through formalization in a product space. Math
  Program 28:96--115

\bibitem[{Schaad(2010)}]{schaad2010modeling}
Schaad J (2010) Modeling the $8$-queens problem and {S}udoku using an algorithm
  based on projections onto nonconvex sets. Master's thesis, University of
  British Columbia

\bibitem[{Svaiter(2011)}]{Svaiter}
Svaiter BF (2011) On weak convergence of the {D}ouglas--{R}achford method. SIAM
  J Control Optim 49(1):280--287

\bibitem[{Tam(2018)}]{tam2018algorithms}
Tam MK (2018) Algorithms based on unions of nonexpansive maps. Optim Letters
  12(5):1019--1027

\bibitem[{Thao(2018)}]{Thao2018}
Thao NH (2018) A convergent relaxation of the {D}ouglas--{R}achford algorithm.
  Comput Optim Appl 70(3):841--863

\end{thebibliography}

\end{document}